%% file: SIAM_article.tex
\documentclass[final,hidelinks,onefignum,onetabnum]{siamart251216}

\input{ex_shared}

\ifpdf
\hypersetup{
  pdftitle={The Lyapunov Equation for Linear Periodic Time Delay Systems},
  pdfauthor={I.V. Aleksandrova, and J.J.L. Vel\'azquez}
}
\fi

\usepackage{amsmath,amsfonts,amssymb}

\renewcommand{\]}{\right]}

\newcommand{\R}{\mathbb{R}}

\newcommand{\C}{\mathbb{C}}

\newcommand{\<}{\leqslant}
\renewcommand{\>}{\geqslant}

\newcommand{\ph}{\varphi}
\newcommand{\eps}{\varepsilon}
\newcommand{\Sum}{\sum\limits}

\renewcommand{\Re}{\mathrm{Re}}

\newcommand{\dd}{\mathrm{d}}
\newcommand{\eqd}{\stackrel{\textup{def}}{=}}
\newcommand{\inner}[2]{\left\langle #1, #2 \right\rangle}

\usepackage{mathrsfs}
\usepackage{graphicx}

\allowdisplaybreaks

\begin{document}

\maketitle

\begin{abstract}
For linear periodic finite-dimensional systems, it is well-known that, first, exponential stability is equivalent to the existence of a unique periodic positive definite solution to the Lyapunov equation, and second, the Lyapunov equation admits a unique periodic solution, if and only if the monodromy matrix has no reciprocal eigenvalues. In the present paper, we generalize these results to the case of periodic evolution families on a Hilbert space, with application to the stability theory of linear periodic systems with constant delays. More precisely, we first link the existence and uniqueness of a quadratic periodic Lyapunov functional with the existence and uniqueness of the solution to a discrete operator Lyapunov equation with the monodromy operator involved. Second, we show that the presented theory on a Hilbert space gives rise to an alternative definition of the delay Lyapunov matrix, the concept previously appeared in the construction of quadratic Lyapunov-Krasovskii functionals for a class of linear periodic delay systems. An explicit connection between the infinite-dimensional Hilbert setting and the previously developed delay Lyapunov matrix framework is established. An important consequence is the uniqueness theorem: the delay Lyapunov matrix exists and is unique, if and only if the monodromy operator has no reciprocal eigenvalues. As a by-product, our framework enables the construction of Lyapunov–Krasovskii functionals for periodic delay systems without a preliminary exponential stability assumption as in earlier theory. 
 \end{abstract}



\section{Introduction} In this paper, we are concerned with the Lyapunov stability theory for a linear periodic time delay system of the form
\begin{equation}
\label{sys_delay_periodic}
\dot x(t) =  A_0(t)x(t) + A_1(t)x(t-h)\quad\text{a.e. for}\;\, t\>t_0,
\end{equation}
where $x\!:\,[t_0-h,+\infty)\to\C^n,$ $A_j:\R\to\R^{n\times n}$ are continuous $T$-periodic matrices, i.e. $A_j(t+T)\equiv A_j(t),$ $\,j=0,1,$ and $h\>0$ is a delay. Although in most of the works initial data belong to the spaces of continuous or piecewise continuous functions, we choose the Hilbert product space $\mathcal H = \mathbb C^n\times L^2([-h,0],\mathbb C^n)$ as the state space following the usual convention in semigroup theory for the time-invariant case. Given $\ph=\bigl(\ph_0,\Phi(\cdot)\bigr)\in\mathcal H,$ let $x(\cdot)=x(\cdot,t_0,\ph)\in H^1([t_0-h,+\infty),\mathbb C^n)$ be a function such that 
\begin{equation}
\label{sys_IVP}
x(t_0)=\ph_0,\quad x(t_0+\xi)=\Phi(\xi)\quad\text{a.e. for}\; \xi\in[-h,0],
\end{equation}
that satisfies equation~(\ref{sys_delay_periodic}) for almost all $t\>t_0.$ In the standard way for time-delay systems, we define the state function as $x_t(t_0,\phi):\,\theta\to x(t+\theta,t_0,\phi),$ $\theta\in[-h,0].$ Then, $\mathrm{x}_t(t_0,\ph) =\bigl(x(t,t_0,\ph),x_t(t_0,\ph)\bigr)\in\mathcal H$ denotes a unique solution to the initial value problem~(\ref{sys_delay_periodic})--(\ref{sys_IVP}). The well-posedness of this problem and the fact that (\ref{sys_delay_periodic})--(\ref{sys_IVP}) defines an evolution family have been justified in \cite{Breda2010}.
Throughout the paper, we assume that $T\>h.$ This does not entail a loss of generality, as the value $kT,$ $k\in\mathbb{Z},$ such that $kT\>h$ can be taken as a period otherwise.

As in the finite-dimensional case, the exponential stability of system~(\ref{sys_delay_periodic}) can be analyzed either via the spectrum of the monodromy operator or through the associated Lyapunov equation. The first group of approaches is known as the Floquet theory, whose extension to the case of delay equations has been discussed in \cite{Stokes1962,Shimanov1963,Zverkin,HaleVerduyn}. It happens that, in contrast to the ODE case, the system of the Floquet solutions is in general not complete \cite{HaleVerduyn} and a reduction of the system to the one with constant coefficients is not always possible. However, it is known that the exponential stability of the system~(\ref{sys_delay_periodic}) is equivalent to the fact that the spectrum of the monodromy operator lies inside the unit circle. There is a number of numerical studies related to approximating the spectrum by discretizing either the monodromy operator or the system of equations itself (see e.g. \cite{Butcher2004,MichielsFenzi2020}).

The second group of approaches extends the second Lyapunov method. It is particularly well developed for the case of the matrices $A_0$ and $A_1$ being constant, where the Lyapunov--Krasovskii functionals with prescribed negative definite derivatives along the solutions of a system deliver necessary and sufficient exponential stability conditions \cite{Repin, InfanteCastelan, Datko1980, Huang, KharZhab2003, Khar_book}. The concept of delay Lyapunov matrix plays a central role in the theory. Remarkably, being just a finite-dimensional object in a single delay case, this matrix allows proving the finite-dimensional stability criteria for time-invariant delay systems \cite{GomEgMond2019,MondieRev2022,AlBel2025}. The fundamental issues such as existence, uniqueness and construction of delay Lyapunov matrices are well addressed \cite{Khar_book}. Similarly to the finite-dimensional case,
the delay Lyapunov equation has found interesting applications in control theory (see e.g. \cite{Jarlebring2011,Gomez2025}).

The first attempt to generalize the delay Lyapunov matrix framework to the case of a periodic system~(\ref{sys_delay_periodic}) appears in \cite{LetZhab2009, ZhabLet2009IFAC}. There, construction of Lyapunov functionals with prescribed derivatives was addressed, the concept of the delay Lyapunov matrix was introduced for exponentially stable systems and its properties have been studied. In \cite{Gomez2016}, the necessary stability conditions based exclusively on the delay Lyapunov matrix, similar to those given in \cite{MondieRev2022} for a time-invariant case, have been proven. The theory was further refined in \cite{Gomez2019_periodic}, with particular emphasis on the computational aspect. The concept of a dual Lyapunov matrix, together with its application to $\mathcal H_2$ analysis and balancing is studied in \cite{MichGom2020}.

The delay Lyapunov matrix framework stems from solving the following problem:
construct a quadratic Lyapunov functional $v_0(t,\ph)$
such that it is differentiable along the solutions of system~(\ref{sys_delay_periodic}) almost everywhere, and
\begin{align}
\label{derivative_v_0}
\dfrac{\dd v_0(t,\mathrm{x}_t(t_0,\ph))}{\dd t} = - x^\star(t,t_0,\ph) W(t) x(t,t_0,\ph)\quad \text{a.e. for}\;\,t\>t_0\>0,
\end{align}
where $W(t)=W^T(t)$ is a real continuous $T$-periodic matrix. This problem is motivated by the Lyapunov--Krasovskii-type stability results \cite[Theorem 1.8]{Khar_book}, as
if $W(t)$ is positive definite for all $t\>0,$ existence of such a functional which satisfies in addition some positive lower and upper bounds implies that system~(\ref{sys_delay_periodic}) is exponentially stable.
The solution to this problem leads to the following

\begin{definition} \textup{\cite{LetZhab2009, Gomez2019_periodic}}
\label{def_Lyap_matrix}
    A matrix $U\in\mathcal{C}(\R^2,\R^{n\times n})$ satisfying the following properties:

    $1.$ PDE property:
\begin{align*}
    &\dfrac{\partial U(\tau,s)}{\partial \tau} = -A_0^T( \tau) U( \tau,s) - A_1^T( \tau+h) U( \tau+h,s),\quad \forall\, s> \tau,
\end{align*}

$2.$ symmetry property:
\begin{align*}
    &U^T( \tau,s)=U(s, \tau),\quad\forall\,  \tau,s\in\R,
    \end{align*}
    
$3.$ ODE property:  
    \begin{align*}
    \dfrac{\textup{d} U( \tau, \tau)}{\textup{d}  \tau} &= -A_0^T( \tau) U( \tau, \tau) - A_1^T( \tau+h) U( \tau+h, \tau)\\ &- U( \tau, \tau) A_0( \tau)
    -U( \tau, \tau+h)A_1( \tau+h)-W(\tau)\quad \forall\,  \tau\in\R,
    \end{align*}

$4.$ periodicity property:
$$
U( \tau+T,s+T)=U( \tau,s),
$$
is called \textbf{the delay Lyapunov matrix} of system~\textup{(\ref{sys_delay_periodic})} associated with $W(\tau).$
\end{definition}
The fundamental result of \cite{LetZhab2009} is that, given a delay Lyapunov matrix, the quadratic functional of the form
\begin{align}
v_0(t_0,\ph) &= \ph_0^\star U(t_0,t_0) \ph_0 +2\Re\left(\ph_0^\star \int_{-h}^0 U(t_0,t_0+\theta+h) A_1(t_0+\theta+h)\Phi(\theta) \textup{d}\theta \right)\notag\\
\label{func_v_0_old_old}
&+  \int_{-h}^0  \int_{-h}^0 \Phi^\star(\theta) A_1^T(t_0+\theta+h) U(t_0+\theta+h,t_0+s+h)\\ &\times A_1(t_0+s+h) \Phi(s) \textup{d}s \textup{d}\theta ,\notag
\end{align}
where $\ph=\bigl(\ph_0,\Phi(\cdot)\bigr)\in\mathcal H,$
precisely satisfies (\ref{derivative_v_0}), which extends the construction of \cite{KharZhab2003} to the periodic case. 
The existence theorem for the delay Lyapunov matrix without assuming exponential stability was recently established in \cite{Egorov2025}. However, the theory still contains significant gaps, primarily concerning uniqueness and construction issues. We note that the computational difficulty comes from the non-standard boundary condition (see the ODE property of Definition~\ref{def_Lyap_matrix}). Moreover, the construction issue is non-trivial even in the case of $h=T,$ where the delay Lyapunov matrix is known to satisfy a hyperbolic PDE system without delay \cite{LetZhab2009}. Some recent advances in this direction can be found in \cite{AlVel_Lyap_matr_scalar,AlVel_Lyap_matr_vector}.

The main contribution of the present paper is the following uniqueness theorem: the delay Lyapunov matrix exists and is unique for any $T$-periodic $W(\tau),$ if and only if the monodromy operator associated to system~\textup{(\ref{sys_delay_periodic})} has no reciprocal eigenvalues. To prove the result, we explore the representation of system~\textup{(\ref{sys_delay_periodic})} as an evolution family in the setting close to \cite{Breda2010}. In Section~\ref{sec_general_Hilbert_setting}, we derive the necessary and sufficient condition for the existence of the functional satisfying an analogue of condition (\ref{derivative_v_0}) for an evolution family. This result explicitly links the existence of the functional with the spectrum of the monodromy operator through solving a discrete operator Lyapunov equation and constitutes a natural extension of the well-known theory for ODEs \cite{BittantiBolzernColaneri1984, BolzernColaneri1988} which is recalled in Section~\ref{sec_ODE}. Despite this, we were not able to find a clear description of this condition in the literature. A close development of the Lyapunov stability theory for the case of general evolution families can be found in \cite{Datko1972}. 
In Section~\ref{sec:main}, we apply the abstract result of Section~\ref{sec_general_Hilbert_setting} to a particular case of system~\textup{(\ref{sys_delay_periodic})}. The major part of this section is devoted to explicitly linking the abstract Hilbert setting and the delay Lyapunov matrix framework described above, a connection that is far from trivial. It is shown that an application of the abstract setting to system~\textup{(\ref{sys_delay_periodic})} gives rise to the concept of the delay Lyapunov matrix as a solution to an integral equation, which is equivalent to those defined above. In fact, the delay Lyapunov matrix already arises while studying the discrete operator Lyapunov equation, where an explicit representation of the monodromy operator is employed. In Section~\ref{sec_existence_uniqueness}, the existence and uniqueness of the delay Lyapunov matrix is proved using the abstract setting. In fact, our approach reveals what actually underlies an abstract operator Lyapunov equation and an abstract Lyapunov functional $\inner{\ph}{\mathcal P(t)\ph}$ for a particular case of system~\textup{(\ref{sys_delay_periodic})}. As a by-product, 
our development allows constructing the functional~(\ref{func_v_0_old_old}) directly from the abstract setting, without a preliminary stability assumption as in \cite{LetZhab2009}.
Our approach provides a different existence proof of the delay Lyapunov matrix independent of those in \cite{Egorov2025} which is based on the analytic continuation. 

\textbf{Notation.} In this paper, $\mathcal H$ denotes a Hilbert space with the scalar product $\inner{\cdot}{\cdot};$  $B(\mathcal H)$ is the space of linear bounded operators acting from $\mathcal H$ to $\mathcal H;$ the space of continuous functions from $A$ to $B$ endowed with the standard uniform norm is denoted by $\mathcal{C}(A,B);$ $L^2(A,B)$ is the space of square integrable functions from $A$ to $B,$ and $H^1(A,B)=\{\psi\in L^2(A,B):\, \psi'\in L^2(A,B)\},$ where the derivative is understood in a weak sense; $\sigma(M)$ stands for the spectrum of a matrix or an operator $M$ depending on the context; $\mathcal P^\star$ is the adjoint operator to $\mathcal P;$ for $x\in\C^n,$ $x^\star=\bar{x}^T,$ where bar denotes the complex conjugation; $\R^n\, (\C^n)$ is the space of real-valued (complex-valued) $n$-dimensional vectors; $\R_+$ stands for the set of nonnegative real numbers.

\section{The ODE periodic Lyapunov equation}
\label{sec_ODE}
As a starting point, we first recall the ODE Lyapunov equation results developed in
\cite{BittantiBolzernColaneri1984, BolzernColaneri1988, DemMatv2001}.
Given a periodic ODE system
\begin{equation}
\label{sys_ODE}
    \dot x(t) = A_0(t)x(t),\quad A_0(t+T)=A_0(t),
\end{equation}
consider the associated boundary-value problem for the Lyapunov equation:
\begin{align}
\label{ODE_Lyapunov}
&\dfrac{\dd P(t)}{\dd t} + A_0^T(t)P(t) + P(t)A_0(t) = -W(t),\quad t\in [0,T],\\
&P(0)=P(T),\notag
\end{align}
where $W(t)=W^T(t)$ is a continuous positive definite matrix such that $W(0)=W(T).$ If $P(t)$ is a solution to~(\ref{ODE_Lyapunov}) on $[0,T]$ then $v(x)=x^T P(t)x$ is a Lyapunov function for system~(\ref{sys_ODE}), with $P(t)$ periodically extended to $\R.$ On the other hand, given the fundamental matrix $\Phi (t,t_0)$ of system~(\ref{sys_ODE}),
\begin{align*}
\dfrac{\partial \Phi (t,t_0)}{\partial t} &= A_0(t)\Phi (t,t_0),\quad
\Phi (t_0,t_0) = I,
\end{align*}
with the properties $\Phi (t,\tau)\Phi (\tau,s)=\Phi (t,s),$ $\,\Phi^{-1}(t,\tau)=\Phi(\tau,t)$ $\,\forall\; t,\tau,s,$  
one can write the general solution to equation~(\ref{ODE_Lyapunov}):
\begin{align}
\label{solution_gen_Lyapunov}
P(t) = \Phi^T (0,t) P(0)\Phi(0,t) - \int_0^t \Phi^T (\xi,t)W(\xi) \Phi (\xi,t) \dd \xi.
\end{align}
An application of the boundary condition $P(0)=P(T)$ brings us to the associated 
discrete Lyapunov (also known as Stein) matrix equation (cf. \cite{Stein}):
\begin{align}
\label{eq_discrete_Lyap}
& P(0) - M^T P(0) M  = W_T,\\
&W_T=\int_0^T \Phi^T (\xi,0)W(\xi) \Phi (\xi,0) \dd \xi,\notag
\end{align}
where $M=\Phi(T,0)$ is the monodromy matrix of system~(\ref{sys_ODE}).
\begin{theorem} \textup{\cite{BittantiBolzernColaneri1984, BolzernColaneri1988}}
\label{thm_ODE}
Given a positive definite matrix $W(t),$ the following holds.

$(i)$ The boundary-value problem~\textup{(\ref{ODE_Lyapunov})} admits a unique solution $P(t),$ equivalently, the matrix equation~\textup{(\ref{eq_discrete_Lyap})} admits a unique solution $P(0),$ if and only if there are no $\lambda,\mu\in\sigma(M)$ such that $\lambda\mu=1.$

$(ii)$ The solution $P(t)$ to \textup{(\ref{ODE_Lyapunov})} is positive definite for all $t,$ equivalently, the solution $P(0)$ to \textup{(\ref{eq_discrete_Lyap})} is positive definite, if and only if $|\lambda|<1$ for any $\lambda\in\sigma(M).$
\end{theorem}

Obviously, the case $(ii)$ corresponds to the exponential stability of system~(\ref{sys_ODE}). In this case,
\begin{align*}
P(0) &= \sum_{k=0}^{+\infty} (M^T)^k W_T M^k = \int_0^{+\infty} \Phi^T (\xi,0)W(\xi) \Phi (\xi,0) \dd \xi, \\
P(t) &= \int_t^{+\infty} \Phi^T (\xi,t)W(\xi) \Phi (\xi,t) \dd \xi
\end{align*}
are positive definite solutions to the matrix equation~(\ref{eq_discrete_Lyap}) and to the boundary-value problem~\textup{(\ref{ODE_Lyapunov})}, respectively. Note that the general solution formula~(\ref{solution_gen_Lyapunov}) contains actually an inverse of the fundamental matrix as $\Phi (\xi,t)=\Phi^{-1} (t,\xi),$ $t\>\xi.$ However, a crucial fact is that using equation~(\ref{eq_discrete_Lyap}) it is possible to rewrite formula~(\ref{solution_gen_Lyapunov}) as
\begin{align}
\label{solution_gen_Lyapunov_modified}
P(t) = \Phi^T (T,t) P(0)\Phi(T,t) + \int_t^T \Phi^T (\xi,t)W(\xi) \Phi (\xi,t) \dd \xi,\quad t\in [0,T],
\end{align}
which is extendable to the infinite-dimensional setting as shown in the next section.

\section{The periodic Lyapunov equation in a Hilbert space}
\label{sec_general_Hilbert_setting}
Here, we present an extension of the result given by Theorem~\ref{thm_ODE} to the case of evolution families on a Hilbert space, where we focus primarily on case $(i)$ of Theorem~\ref{thm_ODE} where no stability assumption is made. Our development is in the spirit of R. Datko's work~\cite{Datko1972}. Note that an equation similar to~(\ref{discr_Lyap}) below appears already in \cite{DK} in the context of exponential dichotomy tests.

Let $\mathcal H$ be a Hilbert space with the scalar product $\langle\cdot,\cdot \rangle$ and the induced norm $\|\cdot\|=\sqrt{\inner{\cdot}{\cdot}}.$
\begin{definition} \textup{\cite{ChiLat}}
\label{def_evolution_family}
A family of mappings $\{S(t,t_0)\}\subset \mathcal B(\mathcal H),$ $0\<t_0\<t<\infty,$ is called a strongly continuous, exponentially bounded evolution family on $\mathcal H,$ if
\begin{align*}
&1.\quad S(t,\xi)S(\xi,t_0)=S(t,t_0),\quad S(t_0,t_0)=I\quad\forall\; 0\<t_0\<\xi\<t,\\ 
&2.\quad\text{the function}\; (t,t_0)\to S(t,t_0)z\;\text{is continuous for all}\;z\in\mathcal H,\; 0\<t_0\<t,\notag\\
&3.\quad \exists\, M\>1,\;\omega>0:\quad\|S(t,t_0)\|\<M e^{\omega(t-t_0)} \quad \forall\;0\<t_0\<t.\notag
\end{align*}
If, in addition, $S(t+T,t_0+T)\equiv S(t,t_0)$ for all $t\>t_0\>0$ then $\{S(t,t_0)\}_{t\>t_0}$ is called a periodic evolution family of period $T.$
\end{definition}
In the sequel, we assume that $S(t,t_0)$ is a strongly continuous, exponentially bounded periodic evolution family of period $T$ on $\mathcal H.$ If there exists $\omega<0$ such that the third condition of Definition~\ref{def_evolution_family} is satisfied then the evolution family $S(t,t_0)$ is called uniformly exponentially stable. The monodromy operator $\mathcal U\!:\mathcal H\to\mathcal H$ is defined as
$$
\mathcal U z=S(T,0)z.
$$
\begin{remark}
    Below, we assume that both $S(t,t_0)$ and $S^\star(t,t_0),$ $t\>t_0,$ act from $\mathcal H$ to $\mathcal H.$ In principle, if one needs to work with the associated abstract Cauchy problem under some well-posedness conditions \textup{\cite{ChiLat}}, the evolution family as a solution operator could also be defined on a smaller domain, see \cite{ChiLat} for details.
\end{remark}
\begin{definition} We say that $\langle z,\mathcal P(t)z \rangle$ is a quadratic $T$-periodic functional on $\mathcal H,$ if $\mathcal P\!:\, \R_+\to \mathcal B(\mathcal H),$ $\mathcal P(t)=\mathcal P^\star(t),$ $\,\mathcal P(t+T)\equiv \mathcal P(t)$ for all $t\>0,$ and the function $t\to\mathcal P(t)z$ is continuous for all $z\in\mathcal H$ and $t\>0.$
\label{def_quadr_functional}
\end{definition}
Consider the function
\begin{align*}
z(t,t_0,z_0)\eqd S(t,t_0)z_0,\quad z_0\in\mathcal H.
\end{align*}

\textbf{Problem~1}. Given a quadratic $T$-periodic functional $\langle z,\mathcal W(t)z \rangle$ on $\mathcal H,$ find a quadratic $T$-periodic functional
\begin{equation*}
    v(t,z) = \langle z, \mathcal P(t)z \rangle,\quad z\in\mathcal H, 
\end{equation*}
such that the function $t\to v(t,z(t,t_0,z_0))$ is differentiable for
$t\>t_0,$ and 
\begin{equation}
\label{func_derivative}
    \dfrac{\dd }{\dd t} v(t,z(t,t_0,z_0)) = -\langle z(t,t_0,z_0), \mathcal W (t) z(t,t_0,z_0) \rangle\;\, \forall \,t\>t_0\>0,\;\,\forall \,z_0\in\mathcal H.
\end{equation}
\begin{theorem}
\label{thm_Hilbert}
Functional $v(t,z) = \langle z, \mathcal P(t)z \rangle$ solves Problem\textup{~1}, if and only if
\begin{align}
\label{P_operator_formula}
\mathcal P(t) &= S^\star(T,t)\mathcal P(0) S(T,t) + \int_t^T S^\star (\xi,t) \mathcal W(\xi) S(\xi,t)\dd \xi,\quad t\in[0,T],\\
\mathcal P(t) &= \mathcal P(t-T)\quad \forall\, t\>T,\notag
\end{align}
where $\mathcal P(0)\in\mathcal B(\mathcal H),$ $\mathcal P(0)=\mathcal P^\star(0),$ is a solution to the discrete operator Lyapunov equation
\begin{align}
\mathcal P(0) - \mathcal U^\star \mathcal P(0)\mathcal U = \int_0^T S^\star(\xi,0)\mathcal W(\xi) S(\xi,0)\dd \xi.\label{discr_Lyap}
\end{align}
\end{theorem}
\begin{proof}
    \textit{Necessity}. Assume that $v(t,z) = \langle z, \mathcal P(t)z \rangle$ is a quadratic $T$-periodic functional which satisfies equality~(\ref{func_derivative}). Let us integrate this equality from $t_0\in[0,T]$ to $T$ and substitute the function $z(t,t_0,z_0)=S(t,t_0)z_0,$ where $ z_0\in\mathcal{H}$ is arbitrary:
$$
\langle S(T,t_0)z_0, \mathcal P(T) S(T,t_0)z_0 \rangle - \langle z_0, \mathcal P(t_0) z_0 \rangle = -\int_{t_0}^{T} \langle S(\xi,t_0)z_0, \mathcal W(\xi) S(\xi,t_0)z_0 \rangle \dd \xi.
$$
For $t_0=0,$ using $\mathcal P(T)=\mathcal P(0)$ and $\mathcal U=S(T,0),$ we get 
$$
\langle\mathcal U z_0, \mathcal P(0)\mathcal U z_0 \rangle - \langle z_0, \mathcal P(0) z_0 \rangle = - \int_0^T \inner{S(\xi,0)z_0}{\mathcal W(\xi) S(\xi,0)z_0}\dd \xi\quad\forall\, z_0\in\mathcal{H},
$$
which is exactly the equation~(\ref{discr_Lyap}).
Now, for $t_0=t\in[0,T],$ with $\mathcal P(T)=\mathcal P(0),$ we obtain
$$
\langle S(T,t)z_0, \mathcal P(0) S(T,t) z_0 \rangle - \langle z_0, \mathcal P(t)z_0 \rangle = -\int_t^T \langle S(\xi,t) z_0, \mathcal W(\xi) S(\xi,t) z_0 \rangle \dd \xi,
$$
for any $t\in[0,T],$ $z_0\in\mathcal{H},$ which coincides with (\ref{P_operator_formula}).

     \textit{Sufficiency}. By definition, for $t\in[kT,(k+1)T]$, $k\in\mathbb Z,$ we have 
 \begin{align*}
\mathcal P(t) &=\mathcal P(t-kT) = S^\star(T,t-kT)\mathcal P(0) S(T,t-kT) \\ &+ \int_{t-kT}^T S^\star (\xi,t-kT) \mathcal W(\xi) S(\xi,t-kT)\dd \xi.
 \end{align*}
Using periodicity of the mappings $S$ and $\mathcal W,$ we get
\begin{align}
\label{formula_P_kt}
\mathcal P(t) &=  S^\star((k+1)T,t)\mathcal P(0) S((k+1)T,t) \\ &+ \int_{t}^{(k+1)T} S^\star (\xi,t) \mathcal W(\xi) S(\xi,t)\dd \xi,\quad t\in[kT,(k+1)T].\notag
\end{align}
     Now, given $k\in\mathbb Z$ and $t\in[kT,(k+1)T]$, consider
     $$
v(t,z(t,t_0,z_0)) =\langle S(t,t_0)z_0, \mathcal P(t)S(t,t_0)z_0 \rangle,\quad t\>t_0\>0,\quad z_0\in\mathcal H.
     $$
Substituting the expression for $\mathcal P(t)$ and using the definition of the evolution family, we further get 
\begin{align*}
    v(t,z(t,t_0,z_0))&= \langle S((k+1)T,t_0)z_0, \mathcal P(0) S((k+1)T,t_0)z_0 \rangle\\
&+ \int_{t}^{(k+1)T} \langle S(\xi,t_0)z_0, \mathcal W(\xi)  S(\xi,t_0)z_0 \rangle \dd \xi,\quad t\in[kT,(k+1)T].
\end{align*}
Since $S(\xi,t_0)z_0$ is continuous by definition for any $z_0\in\mathcal H,$ 
we conclude that the function $v(t,z(t,t_0,z_0))$ is differentiable for $t\>t_0\>0,$ with the derivative equal to (\ref{func_derivative}). 
It remains to mention that for $t=0$ equation~(\ref{P_operator_formula})  is consistent with  (\ref{discr_Lyap}) for which the existence of the solution is assumed.
\end{proof}

\begin{remark}
Clearly, Theorem~\ref{thm_Hilbert} implies that the necessary and sufficient condition for the existence of the Lyapunov functional solving Problem~1 is solvability of the operator equation (\ref{discr_Lyap}). There is no stability assumption in Theorem~\ref{thm_Hilbert}.
It is important to note that in the proof we neither assume the invertibility of the mapping $S(t,t_0)$ nor the differentiability of $\mathcal P(t).$ Moreover, $\mathcal P(t)$ is not always differentiable in a particular case of system~(\ref{sys_delay_periodic}) as can be seen in the next section. We also note that Theorem~\ref{thm_Hilbert} remains true if the equality~(\ref{func_derivative}) is satisfied almost everywhere.
\end{remark}

We now focus on the solvability condition for the operator equation~(\ref{discr_Lyap}) which is known in the literature.

\begin{lemma} \textup{\cite{LumRos1959,Przyluski1980}}
\label{thm_eq_Lyapunov_operator}
If $1\notin \sigma(\mathcal U)\cdot \overline{\sigma(\mathcal U)}$ then the operator equation~\textup{(\ref{discr_Lyap})}
has a unique self-adjoint solution $\mathcal P(0)\in\mathcal B(\mathcal H)$ for any self-adjoint $\mathcal W\!:\,\R_+\to\mathcal B(\mathcal H).$
\end{lemma}
Lemma~\ref{thm_eq_Lyapunov_operator} is a consequence of a general abstract result presented in \cite{LumRos1959}.

\begin{remark}
\label{remark_Lyap_cond}
Lemma~\ref{thm_eq_Lyapunov_operator} implies that if $1\notin \sigma(\mathcal U)\cdot \overline{\sigma(\mathcal U)}$ then the only solution $\mathcal P(0)\in\mathcal B(\mathcal H)$ of the homogeneous operator equation
$$
\mathcal P(0) - \mathcal U^\star \mathcal P(0) \mathcal U = 0
$$
is the trivial one: $\mathcal P(0)z=0$ $\,\forall\, z\in\mathcal H.$
\end{remark}
Denote
$$\mathscr{P} = \{\mathcal W\in \mathcal B(\mathcal H)\;\,|\;\, \mathcal W=\mathcal W^\star,\;\,\exists\,\mu>0:\;\inner{z}{\mathcal W z}\>\mu\|z\|^2 \;\;\forall\, z\in \mathcal H \}.$$
We finally present the following result which is a consequence of Corollary~2 in \cite{Datko1972} and Theorem~\ref{thm_Hilbert}.
\begin{theorem}
\label{thm_Hilbert_stability}
Let $1\notin \sigma(\mathcal U)\cdot \overline{\sigma(\mathcal U)}$ and $\inner{z}{\mathcal W(t) z}$ with $\mathcal W(t)\in \mathscr{P}$ $\,\forall\, t\>0$ be a quadratic $T$-periodic functional. Then,  a $T$-periodic evolution family $S(t,t_0)$ is uniformly exponentially stable, if and only if a unique solution $\mathcal P(0)$ to the operator equation~\textup{(\ref{discr_Lyap})} is nonnegative, i.e. $\inner{z}{\mathcal P(0)z}\>0$ for any $z\in\mathcal H.$
\end{theorem}
\begin{proof}
The fact that $S(t,t_0)$ is uniformly exponentially stable implies that $\sigma(\mathcal U)$ lies inside the unit circle \cite{DK}, and the result follows by Proposition~5 in \cite{Przyluski1980}. Conversely, let a nonnegative $\mathcal P(0)$ be a solution to equation~\textup{(\ref{discr_Lyap})}. By Theorem~\ref{thm_Hilbert}, the functional $v(t,z)=\inner{z}{\mathcal P(t)z}$ with $\mathcal P(t)$ defined
by~(\ref{P_operator_formula}) solves Problem~1. Moreover, $\mathcal P(t)$ is periodic and bounded by construction, hence there exists a constant $k>0$ such that $|v(t,z)|\<k\|z\|^2$ for all $t\>0,$  $z\in\mathcal H.$ In addition, $v(t,z)\>0$ for all $t\>0,$ $z\in\mathcal H.$ 
The uniform exponential stability follows by Corollary~2 in \cite{Datko1972}.
\end{proof}

\begin{remark}
Theorem~\ref{thm_Hilbert} is consistent with the known theory for the exponential stability case.
Indeed, if $S(t,t_0)$ is exponentially stable then $\sigma(\mathcal U)$ lies inside the unit circle \cite{DK}. In this case, a unique solution to equation~(\ref{discr_Lyap}) is given by \cite{Przyluski1980}
$$
\mathcal P(0)=\Sum_{k=0}^{+\infty} \mathcal U^{\star k} \int_0^T S^\star(\xi,0)\mathcal W(\xi) S(\xi,0)\dd \xi\, \mathcal U^k.
$$
Since the periodicity implies $\,\mathcal U^k = S(kT,0)$ and  $S(\xi,0)\mathcal U^k = S(\xi+kT,kT)S(kT,0) = S(\xi+kT,0),$ we conclude that
\begin{align*}
 \mathcal P(0)  &= \Sum_{k=0}^{+\infty} \int_0^T S^\star(\xi+kT,0)\mathcal W(\xi) S(\xi+kT,0)\dd \xi \\ &= \Sum_{k=0}^{+\infty} \int_{kT}^{(k+1)T} S^\star(\xi,0)\mathcal W(\xi-kT) S(\xi,0)\dd \xi= \int_0^{+\infty}S^\star(\xi,0)\mathcal W(\xi) S(\xi,0)\dd \xi.
\end{align*}
By formula~(\ref{formula_P_kt}) we get for $t\in[kT,(k+1)T],$ $k\in\mathbb{Z}:$
\begin{align*}
\mathcal P(t) &=  S^\star((k+1)T,t)\mathcal P(0) S((k+1)T,t) + \int_{t}^{(k+1)T} S^\star (\xi,t) \mathcal W(\xi) S(\xi,t)\dd \xi \\
&= \int_0^{+\infty}S^\star(\xi+(k+1)T,t)\mathcal W(\xi) S(\xi+(k+1)T,t)\dd \xi \\ &+ \int_{t}^{(k+1)T} S^\star (\xi,t) \mathcal W(\xi) S(\xi,t)\dd \xi
=\int_{t}^{+\infty} S^\star (\xi,t) \mathcal W(\xi) S(\xi,t)\dd \xi.
\end{align*}
We arrive at the conclusion that
$$
v(t,z_0) = \langle z_0, \mathcal P(t) z_0 \rangle = \int_t^{+\infty} \langle S (\xi,t) z_0, \mathcal W (\xi) S(\xi,t) z_0 \rangle \dd \xi\quad \forall\,\, t\>0,\; \forall\,\,z_0\in\mathcal H,
$$
which is consistent with \cite{Datko1972} and \cite{LetZhab2009} and could be obtained directly by the integration of (\ref{func_derivative}) from $t_0$ to infinity, taking into account the exponential stability.
\end{remark}

\section{Application of the Hilbert setting to system~(\ref{sys_delay_periodic})}
\label{sec:main}
This section is devoted to an explicit calculation of the operator $\mathcal P(t),$ $t\>0$ given by Theorem~\ref{thm_Hilbert} for a particular case of a periodic time delay system~(\ref{sys_delay_periodic}). We show that, in fact, Theorem~\ref{thm_Hilbert} allows us to recover the Lyapunov functional given by formula~(\ref{func_v_0_old_old}) without a preliminary exponential stability assumption, and discuss consequences delivered by an infinite-dimensional interpretation of the theory.

\subsection{Delay system as an evolution family and Cauchy formula}
Consider a Hilbert space $\mathcal H = \mathbb C^n\times L^2([-h,0],\mathbb C^n)$ with the scalar product
\begin{align*}
\langle\ph,\psi \rangle &= \ph_0^\star \psi_0 + \int_{-h}^0 \Phi^\star(\theta)\Psi(\theta)\dd\theta,\quad \ph,\psi \in\mathcal H,
\end{align*}
where $\ph=\bigl(\ph_0,\Phi(\cdot)\bigr),$ $\psi=\bigl(\psi_0,\Psi(\cdot)\bigr).$
It follows from the existence and uniqueness of solutions to the initial value problem~(\ref{sys_delay_periodic})--(\ref{sys_IVP}) that the associated evolution family
can be defined explicitly as \cite{Breda2010}
\begin{equation}
\label{evolution_family}
S(t,t_0)\ph \eqd \mathrm{x}_t(t_0,\ph) =\bigl( x(t,t_0,\ph),\, x_t(t_0,\ph)\bigr), \quad t\>t_0.
\end{equation}
It is also shown in \cite{Breda2010} that system~(\ref{sys_delay_periodic}) can be written as a well-posed abstract differential equation $\mathrm{x}_t'=\mathcal A(t) \mathrm{x}_t,$ $t\>t_0,$ where the domain of $\mathcal A(t)$ is independent of $t.$ We emphasize however that the evolution family (\ref{evolution_family}) is defined for an arbitrary $\ph\in\mathcal H,$ therefore in this paper we work directly with the evolution family~(\ref{evolution_family}) rather than with an abstract differential equation.
Moreover, the mapping $S(t,t_0)$ can be explicitly represented via 
the fundamental matrix of system~(\ref{sys_delay_periodic}).

\begin{definition} \textup{\cite{Zverkin,HaleVerduyn}}
The matrix $K(t,s)$ is called the fundamental matrix of system~\textup{(\ref{sys_delay_periodic})}, if it satisfies
\begin{align}
\label{fund_matr_t}
\dfrac{\partial K(t,s)}{\partial t} &= A_0(t)K(t,s) + A_1(t)K(t-h,s),\quad t>s, \\
K(s,s)&=I,\quad K(t,s)=\mathbb{O},\quad t<s.\notag
\end{align}
\end{definition}
\begin{remark} \cite{Zverkin,HaleVerduyn}
    It is known that the fundamental matrix satisfies also the equation
    \begin{align}
    \label{fund_matr_s}
\dfrac{\partial K(t,s)}{\partial s} &= -K(t,s)A_0(s)-K(t,s+h) A_1(s+h),\quad t>s,
\end{align}
and is periodic, that is, $K(t+T,s+T)\equiv K(t,s).$
\end{remark}
\begin{lemma} \textup{\cite{Zverkin,HaleVerduyn}} \textup{(\textbf{Cauchy formula})}
    The solution of system~\textup{(\ref{sys_delay_periodic})} with an initial condition $\ph=\bigl(\ph_0,\Phi(\cdot)\bigr)\in\mathcal H$ can be expressed as follows:
\begin{align}
\label{formula_Cauchy}
x(t,t_0,\ph)&= K(t,t_0)\ph_0 + \int_{-h}^0 K(t,t_0+h+\tau)A_1(t_0+h+\tau)\Phi(\tau)\dd\tau,\quad t\>t_0.
\end{align}
\end{lemma}
In fact, the Cauchy formula can be viewed as a definition of the evolution family $S(t,t_0).$ As before, \textit{the monodromy operator} is defined as $\mathcal U:\mathcal H\to \mathcal H,$
\begin{align}
\label{operator_monodromy}
\mathcal U\ph &= S(T,0)\ph = \bigl(
    x(T,0,\ph),\,
    x_T(0,\ph)
\bigr).
\end{align}
Since we assume that $T\>h,$ we can compute the function $x_T(0,\ph)$ explicitly using the Cauchy formula:
\begin{align}
\label{operator_monodromy_2}
x(T+\theta,0,\ph)&= K(T+\theta,0)\ph_0 + \int_{-h}^0 K(T+\theta,h+\tau)A_1(h+\tau)\Phi(\tau)\dd\tau,
\end{align}
where $\theta\in[-h,0].$
The condition $T\>h$ implies that the monodromy operator $\mathcal U$ is compact on $\mathcal{H}$ \cite{Zverkin,HaleVerduyn}, consequently, the nonzero part of the spectrum consists entirely of eigenvalues. The values $\mu\in\sigma(\mathcal U)\setminus\{0\}$ are called \textit{Floquet multipliers} of system~\textup{(\ref{sys_delay_periodic})}.

Next, we present a useful property of the fundamental matrix which is similar to the integral properties of the delay Lyapunov matrix presented in~\cite{Gomez2016}.
\begin{lemma}
\label{lemma_aux_fund}
For any $t\>\xi\>s,$ the fundamental matrix $K(t,s)$ satisfies
\begin{align*}
K(t,s) = K(t,\xi)K(\xi,s) + \int_{-h}^0 K(t,\xi+\theta+h)A_1(\xi+\theta+h)K(\xi+\theta,s)\dd\theta.
\end{align*}
\end{lemma}
\begin{proof}
Let us take $\theta\in(s,t)$ and consider the expression
\begin{align*}
\dfrac{\partial}{\partial\theta}(K(t,\theta)K(\theta,s)) &= -(K(t,\theta)A_0(\theta)+K(t,\theta+h) A_1(\theta+h))K(\theta,s)\\
&+K(t,\theta)(A_0(\theta)K(\theta,s) + A_1(\theta)K(\theta-h,s))\\
&= -K(t,\theta+h) A_1(\theta+h)K(\theta,s) + K(t,\theta)A_1(\theta)K(\theta-h,s).
\end{align*}
Now, we integrate both sides of the above equality from $s$ to $\xi\in(s,t):$
\begin{align*}
&K(t,\xi)K(\xi,s) - K(t,s) \\ &= -\int_{s}^\xi K(t,\theta+h) A_1(\theta+h)K(\theta,s)\dd\theta + \int_{s}^\xi K(t,\theta)A_1(\theta)K(\theta-h,s)\dd\theta\\
&=-\int_{s}^\xi K(t,\theta+h) A_1(\theta+h)K(\theta,s)\dd\theta + \int_{s-h}^{\xi-h} K(t,\theta+h)A_1(\theta+h)K(\theta,s)\dd\theta.
\end{align*}
Since $K(\theta,s)=\mathbb{O}$ for $\theta<s,$ we can extend the lower limit of the first integral to $s-h,$ hence we get
\begin{align*}
K(t,s) = K(t,\xi)K(\xi,s) + \int_{\xi-h}^\xi K(t,\theta+h) A_1(\theta+h)K(\theta,s)\dd\theta,\quad t\>\xi\>s,
\end{align*}
which coincides with the desired relation.
\end{proof}
\begin{remark} In fact, Lemma~\ref{lemma_aux_fund} is a consequence of $S(t,\xi)S(\xi,s)=S(t,s)$ for any $t\>\xi\>s.$
\end{remark}

\subsection{Derivation of $\mathcal P(0)$}
\label{sec_P_0}
The aim of this section is to solve explicitly the discrete Lyapunov equation~(\ref{discr_Lyap}),
\begin{align}
\label{discr_Lyap_ph}
&\langle \ph, \mathcal P(0) \ph \rangle - \langle\mathcal U \ph, \mathcal P(0)\mathcal U \ph \rangle = \int_0^T \inner{S(\xi,0)\ph}{\mathcal W(\xi) S(\xi,0)\ph}\dd \xi,\quad \ph\in\mathcal H,
\end{align}
for a particular case of system~(\ref{sys_delay_periodic}), by exploiting the structure of the mappings $\mathcal U$ and $S(\xi,0)$ given by the Cauchy formula. Motivated by the structure~(\ref{derivative_v_0}) of the derivative of the functional $v_0$ presented in the introduction, in the subsequent development we consider the operator $\mathcal W\!:\, \R_+\to\mathcal B(\mathcal H)$ of the form
\begin{equation}
\label{formula_W}
\mathcal W(t)\ph = \begin{pmatrix}
        W(t) \ph_0 \\ 0
    \end{pmatrix},
\end{equation}
where $W(t)=W^T(t)$ is a real continuous $T$-periodic matrix.
\begin{remark}
    \label{remark_W}
    The operator $\mathcal W$ given by (\ref{formula_W}) is obviously not in $\mathscr{P},$ and hence it does not satisfy the conditions of Theorem~\ref{thm_Hilbert_stability}. However, we will show that considering~(\ref{formula_W}) is a necessary first step to connect the delay Lyapunov matrix framework with the infinite-dimensional setting of Section~\ref{sec_general_Hilbert_setting} which suffices for the aims of this paper. Unfortunately, the framework of complete type functionals \cite{Khar_book, LetZhab2009} is not well suited to a Hilbert space setting due to dependence of the functional derivative on $x(t-h).$ Therefore, a modification of $\mathcal W$ is required in order to construct a functional satisfying the conditions of Theorem~\ref{thm_Hilbert_stability}.
\end{remark}

We start substituting the Cauchy formula~(\ref{formula_Cauchy}) in the right-hand side of equation~(\ref{discr_Lyap_ph}):
\begin{align*}
\int_0^T &\inner{S(\xi,0)\ph}{\mathcal W(\xi) S(\xi,0)\ph}\dd \xi = \int_0^T x^\star(\xi,0,\ph)W(\xi) x(\xi,0,\ph)\dd \xi \\
&= \int_0^T\left[K(\xi,0)\ph_0 + \int_{-h}^0 K(\xi,h+\tau)A_1(h+\tau)\Phi(\tau)\dd\tau\right]^\star W(\xi) \\ &\times\left[K(\xi,0)\ph_0 + \int_{-h}^0 K(\xi,h+\tau)A_1(h+\tau)\Phi(\tau)\dd\tau\right]\dd \xi.
\end{align*}
Expanding the brackets, we get the following representation of the right-hand side of the equation:
\begin{align}
\label{struct_Lyap_eq}
\int_0^T &\inner{S(\xi,0)\ph}{\mathcal W(\xi) S(\xi,0)\ph}\dd \xi =\ph_0^\star Y_0(-h,-h) \ph_0\\ &+2\Re\left(\ph_0^\star \int_{-h}^0 Y_0(-h,\tau)A_1(h+\tau)\Phi(\tau)\dd\tau\right) \notag\\
&+\int_{-h}^0\int_{-h}^0 \Phi^\star(\tau_1)A_1^T(h+\tau_1) Y_0 (\tau_1,\tau_2) A_1(h+\tau_2)\Phi(\tau_2)\dd\tau_2 \dd\tau_1\notag
\end{align}
with
\begin{align*}
Y_0 (\tau_1,\tau_2) &= \int_0^T K^T(\xi,h+\tau_1) W(\xi) K(\xi,h+\tau_2)\dd \xi,\quad \tau_1,\tau_2\in[-h,0].
\end{align*}
The form~(\ref{struct_Lyap_eq}) of the right-hand side of equation~(\ref{discr_Lyap_ph}) naturally suggests us seeking for its solution $\mathcal P(0):\,\mathcal H\to \mathcal H\,$ with the following structure:
\begin{align}
\notag
 \langle\ph, \mathcal P(0)\ph \rangle &= \ph_0^\star M_0\ph_0 + 2\Re \left(\ph_0^\star\int_{-h}^0 Q_0(\theta)A_1(h+\theta)\Phi(\theta)\dd \theta \right) \\ &+ \int_{-h}^0\int_{-h}^0\Phi^\star(\theta)A_1^T(h+\theta)R_0(\theta,s)A_1(h+s)\Phi(s)\dd s\dd \theta,\label{structure_P_0}
\end{align}
where
\begin{equation}
\label{cond_P_0_self}
M_0=M_0^T,\quad R_0(\theta,s) = R_0^T(s,\theta).
\end{equation}
The conditions~(\ref{cond_P_0_self}) come from the self-adjointness of $\mathcal P(0).$ We emphasize that the structure of the functional is suggested by the Cauchy formula in the same way as in the previous theory \cite{LetZhab2009, ZhabLet2009IFAC} but without assuming the exponential stability.

Based on the structure~(\ref{structure_P_0}), we now compute the term $\langle\mathcal U\ph, \mathcal P(0)\mathcal U\ph \rangle$ using the expressions~(\ref{operator_monodromy}) and (\ref{operator_monodromy_2}) for the monodromy operator. After expanding the brackets we get the following representation:
\begin{align*}
\langle\mathcal U\ph, \mathcal P(0)\mathcal U\ph \rangle &=\ph_0^\star M_1\ph_0 + 2\Re \left(\ph_0^\star\int_{-h}^0 Q_1(\theta)A_1(h+\theta)\Phi(\theta)\dd \theta \right) \\ &+ \int_{-h}^0\int_{-h}^0\Phi^\star(\theta)A_1^T(h+\theta)R_1(\theta,s)A_1(h+s)\Phi(s)\dd s\dd \theta,
\end{align*}
where, with the notation
\begin{align*}
\Psi(\theta,s; M_0,Q_0,R_0)&= K^T(T,h+\theta) M_0 K(T,h+s)\\ &+ K^T(T,h+\theta)\int_{-h}^0 Q_0(\xi)A_1(h+\xi)K(T+\xi,h+s)\dd\xi\\
&+ \int_{-h}^0 K^T(T+\xi,h+\theta)A_1^T(h+\xi)Q_0^T(\xi)\dd\xi K(T,h+s)\\
&+ \int_{-h}^0\int_{-h}^0 K^T(T+\xi_1,h+\theta)A_1^T(h+\xi_1)R_0(\xi_1,\xi_2)\\ &\times A_1(h+\xi_2)K(T+\xi_2,h+s) \dd \xi_2\dd \xi_1,\quad\theta,s\in[-h,0],
\end{align*}
we have
\begin{align*}
M_1 &=\Psi(-h,-h; M_0,Q_0,R_0),\\
Q_1(\theta) &=\Psi(-h,\theta; M_0,Q_0,R_0),\\
R_1(\theta,s) &=\Psi(\theta,s; M_0,Q_0,R_0).
\end{align*}
From here, we immediately conclude that
$$
M_1=R_1(-h,-h),\quad Q_1(\theta) = R_1(-h,\theta).
$$
We now consider equation~(\ref{discr_Lyap_ph}). Combining the terms of the same kind, it is easy to see that this operator equation gives us three matrix equations:
\begin{align*}
M_0 &= R_1(-h,-h) + Y_0(-h,-h),\\
Q_0(\theta) &=R_1(-h,\theta) + Y_0(-h,\theta),\\
R_0(\theta,s) &= R_1(\theta,s)+ Y_0(\theta,s).
\end{align*}
This, in turn, implies that
\begin{equation}
\label{M_Q_0}
M_0=R_0(-h,-h),\quad Q_0(\theta) = R_0(-h,\theta),
\end{equation}
and thus the first two equations are particular cases of the third one. We conclude that the operator equation~(\ref{discr_Lyap_ph}) gives us the matrix integral equation for $R_0,$
\begin{align*}
&R_0(\theta,s) = K^T(T,h+\theta) R_0(-h,-h) K(T,h+s)\label{eq_for_R}\\ &+ K^T(T,h+\theta)\int_{-h}^0 R_0(-h,\xi)A_1(h+\xi)K(T+\xi,h+s)\dd\xi\notag\\
&+ \int_{-h}^0 K^T(T+\xi,h+\theta)A_1^T(h+\xi)R_0(\xi,-h)\dd\xi K(T,h+s)\notag\\
&+ \int_{-h}^0\int_{-h}^0 K^T(T+\xi_1,h+\theta)A_1^T(h+\xi_1)R_0(\xi_1,\xi_2)A_1(h+\xi_2)\notag\\ &\times K(T+\xi_2,h+s) \dd \xi_2\dd \xi_1
+\int_{0}^T K^T(\tau,h+\theta) W(\tau) K(\tau,h+s)\dd \tau,\; \theta,s\in[-h,0],\notag
\end{align*}
together with the conditions~(\ref{M_Q_0}) and the symmetry assumption $R_0(\theta,s)=R_0^T(s,\theta).$
For further convenience, we now make a shift of arguments and define a new function
$$
U_0(\theta,s)=R_0(\theta-h,s-h),\quad \theta,s\in[0,h].
$$
The obtained result is summarized in the following lemma.
\begin{lemma}
If $\,U_0\in\mathcal C([0,h]^2,\R^{n\times n})$ solves an integral equation
\begin{align}
\label{eq_integral_U_0}
    U_0(\theta,s) &= [\mathcal L U_0] (\theta,s)+\int_0^{T} K^T(\xi,\theta) W(\xi) K(\xi,s)\dd \xi,\quad \theta,s\in[0,h],
\end{align}
where $\,\mathcal L\!:\, \mathcal C([0,h]^2,\R^{n\times n})\to \mathcal C([0,h]^2,\R^{n\times n})$ is defined as
\begin{align}
\label{operator_L}
&[\mathcal L U_0] (\theta,s)=K^T(T,\theta) U_0(0,0) K(T,s) \\ &+ K^T(T,\theta)\int_{-h}^0 U_0(0,h+\xi)A_1(h+\xi)K(T+\xi,s)\dd\xi\notag\\
&+ \int_{-h}^0 K^T(T+\xi,\theta)A_1^T(h+\xi)U_0(\xi+h,0)\dd\xi K(T,s)\notag\\
&+ \int_{-h}^0\!\int_{-h}^0\! K^T(T+\xi_1,\theta)A_1^T(h+\xi_1)U_0(h+\xi_1,h+\xi_2)\notag\\ &\times A_1(h+\xi_2) K(T+\xi_2,s) \dd \xi_2\dd \xi_1,\notag 
\end{align}
then the operator $\mathcal P(0)\in\mathcal B(\mathcal H)$ defined by
\begin{align}
\label{formula_P_0}
\mathcal P(0)\ph =
\begin{pmatrix}
    U_0(0,0)\ph_0 + \int_{-h}^0 U_0(0,h+\theta)A_1(h+\theta)\Phi(\theta)\dd \theta \\
    A_1^T(h+\cdot)\[U_0(h+\cdot,0)\ph_0 +  \int_{-h}^0U_0(h+\cdot,h+s)A_1(h+s)\Phi(s)\dd s\]
\end{pmatrix}
\end{align}
solves the operator Lyapunov equation~\textup{(\ref{discr_Lyap_ph})} with $\mathcal W$ given by \textup{(\ref{formula_W})}.
\end{lemma}

\subsection{Derivation of $\mathcal P(t)$}
\label{sec_P_t}
In this section, we assume that the solution $\,U_0\in\mathcal C([0,h]^2,\R^{n\times n})$ to the integral equation~(\ref{eq_integral_U_0}) exists. Having the expression for $\mathcal P(0)$ defined by (\ref{formula_P_0}) at hand, we can now compute 
the operator $\mathcal P(t)\!:\,[0,T]\to \mathcal B(\mathcal H)$ given by formula~(\ref{P_operator_formula}) for our particular problem.
We start with the second term in~(\ref{P_operator_formula}) substituting there the Cauchy formula~(\ref{formula_Cauchy}):
\begin{align}
\int_t^T \langle S &(\xi,t) \ph, \mathcal W(\xi) S(\xi,t) \ph \rangle \dd \xi = \int_t^T x^\star(\xi,t,\ph)W(\xi) x(\xi,t,\ph)\dd \xi \notag\\
&=\int_t^T \[K(\xi,t)\ph_0 + \int_{-h}^0 K(\xi,t+h+\tau_1)A_1(t+h+\tau_1)\Phi(\tau_1)\dd\tau_1\]^\star W(\xi) \notag\\
&\times  \[K(\xi,t)\ph_0 + \int_{-h}^0 K(\xi,t+h+\tau_2)A_1(t+h+\tau_2)\Phi(\tau_2)\dd\tau_2\]\dd \xi\notag
\\
&=\ph_0^\star Y(t,-h,-h)\ph_0\label{term_2_P_t}\\
&+2\Re\left(\ph_0^\star\int_{-h}^0 Y(t,-h,\theta)A_1(t+h+\theta)\Phi(\theta)\dd\theta\right)\notag\\
&+\int_{-h}^0\int_{-h}^0 \Phi^\star (\theta) A_1^T(t+h+\theta) Y(t,\theta,s) A_1(t+h+s)\Phi(s)\dd s\dd\theta\notag,
\end{align}
where
$$
Y(t,\theta,s) = \int_t^T K^T(\xi,t+h+\theta) W(\xi)  K(\xi,t+h+s) \dd \xi,\quad t\in[0,T],\quad \theta,s\in[-h,0],
$$
and $Y(0,\theta,s)=Y_0(\theta,s).$ Now, let us turn our attention to computing the first term in~(\ref{P_operator_formula}),
$$
\langle S(T,t)\ph, \mathcal P(0) S(T,t) \ph \rangle,\quad t\in [0,T].
$$
To do this, we need $S(T,t)\ph = \bigl(
    x(T,t,\ph),\,
    x_T(t,\ph)
\bigr),$ where, according to the Cauchy formula,
\begin{align}
\label{Cauchy_P_t_1}
x(T+\theta,t,\ph)&= K(T+\theta,t)\ph_0 + \int_{-h}^0 \!\! K(T+\theta,t+\tau+h)A_1(t+\tau+h)\Phi(\tau)\dd\tau,
\end{align}
if $T+\theta\>t,$ and 
\begin{align*}
x(T+\theta,t,\ph)&=\Phi(T+\theta-t),\quad t-h\<T+\theta<t.
\end{align*}
Here, $\theta\in[-h,0]$ and $\,t\in[0,T].$
Then,
\begin{align*}
\inner{S(T,t)\ph}{\mathcal P(0)S(T,t)\ph} = I_1 + (I_2 + \Lambda_2) + (I_3+ \Lambda_3),
\end{align*}
where
\begin{align*}
I_1 &= x^\star(T,t,\ph) U_0(0,0) x(T,t,\ph),\\
I_2 &=2\Re\left(x^\star(T,t,\ph)\int_{\max\{t-T,-h\}}^0 U_0(0,h+\theta)A_1(h+\theta) x(T+\theta,t,\ph)\dd \theta\right), \\
I_3 &= \int_{\max\{t-T,-h\}}^0\int_{\max\{t-T,-h\}}^0 x^\star(T+\theta,t,\ph) A_1^T(h+\theta) U_0(h+\theta,h+s)\\ &\times A_1(h+s) x(T+s,t,\ph)\dd s\dd \theta,\notag
\end{align*}
and $\Lambda_2=\Lambda_3=0,$ if $t-T\<-h,$ and
\begin{align*}
\Lambda_2&=2\Re\left(x^\star(T,t,\ph)\int_{-h}^{t-T} U_0(0,h+\theta)A_1(h+\theta) \Phi(T+\theta-t)\dd \theta\right),\\
\Lambda_3&= \int_{-h}^{t-T}\int_{t-T}^0 \Phi^\star(T+\theta-t) A_1^T(h+\theta) U_0(h+\theta,h+s) A_1(h+s)\\ &\times x(T+s,t,\ph)\dd s\dd \theta\\
&+\int_{t-T}^0 \int_{-h}^{t-T} x^\star(T+\theta,t,\ph) A_1^T(h+\theta) U_0(h+\theta,h+s)A_1(h+s) \\ &\times  \Phi(T+s-t)\dd s\dd \theta \\
&+ \int_{-h}^{t-T}\int_{-h}^{t-T} \Phi^\star(T+\theta-t) A_1^T(h+\theta) U_0(h+\theta,h+s) A_1(h+s)\\ &\times  \Phi(T+s-t)\dd s\dd \theta,
\end{align*}
if $t-T\>-h.$ Here, the integrals are split in such a way that $x(T+\cdot,t,\ph)$ always corresponds to the expression~(\ref{Cauchy_P_t_1}). Substituting this expression, expanding the brackets, and taking into account the term~(\ref{term_2_P_t}), after some tedious calculations we get the following structure:
\begin{align}
v_0(t,\ph) &= \langle\ph, \mathcal P(t)\ph \rangle = \ph_0^\star M(t)\ph_0 + 2\Re \left(\ph_0^\star\int_{-h}^0 Q(t,\theta)A_1(t+h+\theta)\Phi(\theta)\dd \theta \right)\notag \\ &+ \int_{-h}^0\int_{-h}^0\Phi^\star(\theta)A_1^T(t+h+\theta)R(t,\theta,s)A_1(t+h+s)\Phi(s)\dd s\dd \theta + \Lambda, \label{formula_v_t}
\end{align}
where $\,\Lambda=\Lambda_2+\Lambda_3,$ $\,\alpha(t)=\max\{t-T,-h\},$ and
\begin{align}
M(t) &= R(t,-h,-h), \notag\\  Q(t,\theta) &= R(t,-h,\theta), \notag\\
R(t,\theta,s) &=  K^T(T,h+t+\theta)U_0(0,0) K(T,h+t+s)\notag\\
&+ K^T(T,t+\theta+h) \int_{\alpha(t)}^0 U_0(0,h+\xi) A_1(h+\xi) K(T+\xi,t+s+h)\dd\xi \label{eq_R_t}\\
&+ \int_{\alpha(t)}^0 K^T(T+\xi,t+\theta+h) A_1^T(h+\xi) U_0(h+\xi,0) \dd\xi K(T,t+s+h)\notag\\
&+\int_{\alpha(t)}^0\int_{\alpha(t)}^0 K^T(T+\xi_1,t+\theta+h) A_1^T(h+\xi_1) U_0(h+\xi_1,h+\xi_2) A_1(h+\xi_2)\notag \\
&\times K(T+\xi_2,t+s+h)\dd \xi_2\dd \xi_1 +Y(t,\theta,s),\quad \theta,s\in[-h,0].\notag
\end{align}
Equation~(\ref{eq_R_t}) serves as a definition of the matrix function $R(t,\theta,s)$ when the function $U_0\in\mathcal C([0,h]^2,\R^{n\times n}) $ is assumed known.
From now on, let us distinguish between the cases when $t\in[0,T-h]$ and $t\in[T-h,T].$

\textbf{Case~1: $t\in[0,T-h]$.} In this case, $\alpha(t)=-h$ and $\Lambda=0.$ When $t\<T-h,$ the function $R(t,\theta,s)$ is defined by formula~(\ref{eq_R_t}) for any $\theta,s\in[-h,0].$
The desired Lyapunov functional is then given by formula~(\ref{formula_v_t}) with $\Lambda=0,$ $\,M(t) = R(t,-h,-h),$ $\,Q(t,\theta)=R(t,-h,\theta).$

\textbf{Case~2: $t\in[T-h,T]$.} In this case, $\alpha(t)=t-T,$ however, we first notice that $\alpha(t)$ in~(\ref{eq_R_t}) can be still replaced with $-h.$ Indeed, the integrand in the first integral is only non-zero when $\xi\>t+s+h-T.$ Since $t+s+h-T\>t-T\>-h,$ we can extend the lower integration limit to $-h.$ The other integrals can be treated similarly. Second, we mention that the function $R(t,\theta,s)$ defined by formula~(\ref{eq_R_t}) is zero if either $\theta>T-h-t$ or $s>T-h-t.$ Hence, the functional~(\ref{formula_v_t}) in this case is in fact of the form
\begin{align*}
\langle\ph, &\mathcal P(t)\ph \rangle = \ph_0^\star R(t,-h,-h)\ph_0 + 2\Re \left(\ph_0^\star\int_{-h}^{T-h-t} R(t,-h,\theta)A_1(t+h+\theta)\Phi(\theta)\dd \theta \right)\notag \\ &+ \int_{-h}^{T-h-t}\int_{-h}^{T-h-t}\Phi^\star(\theta)A_1^T(t+h+\theta)R(t,\theta,s)A_1(t+h+s)\Phi(s)\dd s\dd \theta + \Lambda. 
\end{align*}
We now proceed with the computation of $\Lambda=\Lambda_2+\Lambda_3:$
\begin{align*}
\Lambda_2 &=2\Re\Biggl(\left[K(T,t)\ph_0 + \int_{-h}^0 K(T,t+\tau+h)A_1(t+\tau+h)\Phi(\tau)\dd\tau\right]^\star\\ &\times \int_{T-h-t}^{0} U_0(0,h+\xi+t-T)A_1(h+\xi+t) \Phi(\xi)\dd \xi\Biggr),\\
\Lambda_3 &= 
2\Re\;\ph_0^\star \int_{t-T}^0 K^T(T+\theta,t) A_1^T(h+\theta) \\ &\times \int_{T-h-t}^{0}  U_0(h+\theta,h+t+\xi_2-T) A_1(h+t+\xi_2) \Phi(\xi_2)\dd \xi_2\dd \theta\\
&+\int_{T-h-t}^{0} \Phi^\star(\xi_1) A_1^T(h+t+\xi_1) \int_{t-T}^0  U_0(h+t+\xi_1-T,h+s) A_1(h+s)\\
&\times \int_{-h}^0 K(T+s,t+\tau_2+h)A_1(t+\tau_2+h)\Phi(\tau_2)\dd\tau_2\dd s\dd \xi_1\\
&+\int_{-h}^0 \Phi^\star(\tau_1) A_1^T(t+\tau_1+h) \int_{t-T}^0 K^T(T+\theta,t+\tau_1+h)A_1^T(h+\theta)   \\ &\times \int_{T-h-t}^{0}  U_0(h+\theta,h+t+\xi_2-T) A_1(h+t+\xi_2) \Phi(\xi_2)\dd \xi_2\dd \theta \dd\tau_1,\\
&+ \int_{T-h-t}^{0}\int_{T-h-t}^{0} \Phi^\star(\xi_1) A_1^T(h+t+\xi_1) U_0(h+t+\xi_1-T,h+t+\xi_2-T)\\ &\times A_1(h+t+\xi_2) \Phi(\xi_2)\dd \xi_2\dd \xi_1,
\end{align*}
where the periodicity of $A_1$ is used. Let us first collect the terms of the structure $2\Re\, \ph_0^\star \int \ldots \Phi(\xi)\dd \xi$ in $\Lambda:$
\begin{align*}
2&\Re\; \ph_0^\star  \int_{T-h-t}^{0} \biggl[ K^T(T,t) U_0(0,h+\xi+t-T) \\ &+ \int_{t-T}^0 K^T(T+\tau,t)A_1^T(h+\tau)U_0(h+\tau,h+\xi+t-T)\dd\tau \biggr] A_1(h+\xi+t) \Phi(\xi)\dd \xi.
\end{align*}
Similarly, we combine the remaining terms:
\begin{align*}
&\int_{T-h-t}^{0}\dd \xi\; \Phi^\star(\xi)A_1^T(h+\xi+t)\int_{-h}^0 \dd\tau_2 
\biggl[U_0(h+\xi+t-T,0)  K(T,h+t+\tau_2)\\ &+\int_{t-T}^0 \!\!\! U_0(h+t+\xi-T,h+s) A_1(h+s) K(T+s,t+\tau_2+h)\dd s \biggr] A_1(h+t+\tau_2)\Phi(\tau_2)\\
&+\int_{-h}^0\dd\tau_1\; \Phi^\star(\tau_1) A_1^T(t+\tau_1+h) \int_{T-h-t}^{0}\dd \xi \biggl[ K^T(T,h+t+\tau_1) U_0(0,h+\xi+t-T) \\ &+ \int_{t-T}^0\!\!\!\! K^T(T+\theta,t+\tau_1+h)A_1^T(h+\theta)  U_0(h+\theta,h+t+\xi-T) \dd \theta \biggr] A_1(h+\xi+t) \Phi(\xi)\\
&+ \int_{T-h-t}^{0}\int_{T-h-t}^{0} \Phi^\star(\xi_1) A_1^T(h+t+\xi_1) U_0(h+t+\xi_1-T,h+t+\xi_2-T)\\ &\times A_1(h+t+\xi_2) \Phi(\xi_2)\dd \xi_2\dd \xi_1.
\end{align*}
Note that the integrals with respect to $\tau_1$ and $\tau_2$ here are only from $-h$ to $T-h-t,$ since the fundamental matrix inside is equal to zero otherwise. We conclude that, to obtain the same structure of the functional $v_0(t,\ph)$ as in Case~1, we have to define the function $R$ as follows:
\begin{align*}
R(t,\xi_1,\xi_2)&=U_0(h+\xi_1+t-T,0)  K(T,h+t+\xi_2)\\ &+\int_{t-T}^0 \!\!\! U_0(h+t+\xi_1-T,h+s) A_1(h+s) K(T+s,t+\xi_2+h)\dd s,\\ &\phantom{asdasfasdfafasdfasdfas}\xi_1\in[T-h-t,0],\,\xi_2\in[-h,T-h-t],\\
R(t,\xi_1,\xi_2)&=K^T(T,h+t+\xi_1) U_0(0,h+\xi_2+t-T) \\ &+ \int_{t-T}^0\!\!\!\! K^T(T+\theta,t+\xi_1+h)A_1^T(h+\theta)  U_0(h+\theta,h+t+\xi_2-T) \dd \theta,\\ &\phantom{asdasfasdfafasdfasdfas}\xi_1\in[-h,T-h-t],\,\xi_2\in[T-h-t,0],\\
R(t,\xi_1,\xi_2)&=U_0(h+t+\xi_1-T,h+t+\xi_2-T),\quad \xi_1,\xi_2\in[T-h-t,0],
\end{align*}
and for $\xi_1,\xi_2\in[-h,T-h-t]$
the function $R(t,\xi_1,\xi_2)$ is still defined by formula~(\ref{eq_R_t}) with $\alpha(t)$ replaced by $-h.$ 
We conclude that the Lyapunov functional in both cases has the form
\begin{align}
v_0(t,\ph) &= \langle\ph, \mathcal P(t)\ph \rangle = \ph_0^\star R(t,-h,-h)\ph_0 \notag \\ &+ 2\Re \left(\ph_0^\star\int_{-h}^{0} R(t,-h,\theta)A_1(t+h+\theta)\Phi(\theta)\dd \theta \right)\label{formula_v_t_R}\\ &+ \int_{-h}^{0}\int_{-h}^{0}\Phi^\star(\theta)A_1^T(t+h+\theta)R(t,\theta,s)A_1(t+h+s)\Phi(s)\dd s\dd \theta, \notag 
\end{align}
with the matrix function $R(t,\theta,s)$ defined as above.

\begin{remark}
     Whereas the structure of the operator $\mathcal P(0)$ as in (\ref{structure_P_0}) is an assumption, although very natural and coming from the structure of the solution operator, the obtained structure of $\mathcal P(t)$ is a consequence of those structure of $\mathcal P(0).$
\end{remark}
 
\subsection{The function $R(t,\theta,s)$ as an extension of $U_0$}
In this section, let us take a closer look at the definition of the function $R$ adopted above. We start with considering the integral equation~(\ref{eq_R_t}). First of all, we notice that
\begin{align*}
Y(t,\theta,s) &= \int_t^T K^T(\xi,t+h+\theta) W(\xi)  K(\xi,t+h+s) \dd \xi \\ &= \int_{t+h+\theta}^T K^T(\xi,t+h+\theta) W(\xi)  K(\xi,t+h+s) \dd \xi,
\end{align*}
if $t+h+\theta\<T,$ and $Y(t,\theta,s) = 0$ otherwise. This implies that in fact
$$
Y(t,\theta,s) \eqd \widetilde{Y}(t+\theta+h,t+s+h).
$$
But now equation~(\ref{eq_R_t}) tells us that $R$ is actually a function of two variables, namely, 
$$
R(t,\theta,s)\eqd\widetilde{U}(t+\theta+h,t+s+h).
$$
We now notice that the right-hand side of equation~(\ref{eq_R_t}) in the new variables is represented by an extension of the same operator $\mathcal L$ defined in (\ref{operator_L}), namely, $$\,\mathcal L\!:\, \mathcal C([0,h]^2,\R^{n\times n})\to \mathcal C([0,T]^2,\R^{n\times n}).$$ 
The definition of the function $R$ can be then written in terms of $\widetilde{U}$ as follows (see Figure~\ref{Fig_1} for the illustration):

$(A).$ For $(\xi,\eta)\in \bigl\{0\<\xi,\eta\<T\bigr\}\bigcap\bigl\{\xi-h\<\eta\<\xi+h\bigr\},$
\begin{align*}
\widetilde{U}(\xi,\eta)  = [\mathcal L U_0 ](\xi,\eta) + \int_{\xi}^T K^T(\tau,\xi) W(\tau)  K(\tau,\eta) \dd \tau,
\end{align*}

$(B).$ For
$
(\xi,\eta)\in \bigl\{T-h\<\xi-h\<\eta\<T\bigr\},
$
\begin{align*}
\widetilde{U}(\xi,\eta)  = U_0(\xi-T,0)K(T,\eta) + \int_{-h}^0 U_0(\xi-T,h+\tau)A_1(h+\tau)K(T+\tau,\eta)\dd\tau,
\end{align*}

$(C).$ For
$(\xi,\eta)\in \bigl\{T\<\eta\<\xi+h\<T+h\bigr\},$ $\widetilde{U}(\xi,\eta)=\widetilde{U}^T(\eta,\xi),$ where $\widetilde{U}(\eta,\xi)$ is defined in $(B),$

$(D).$ For $(\xi,\eta)\in \bigl\{T\<\xi,\eta\<T+h\bigr\},$ the function is defined by periodicity,
\begin{align*}
\widetilde{U}(\xi,\eta)  = U_0(\xi-T,\eta-T).
\end{align*}
The symmetry property $\widetilde{U}(\xi,\eta) = \widetilde{U}^T(\eta,\xi)$ follows from $R(t,\theta,s)=R^T(t,s,\theta).$ Here, 
the condition $\xi,\eta\in[0,T]$ in the case $(A)$ comes from $\theta,s\in[-h,T-h-t],$ whereas $\xi-h\<\eta\<\xi+h$ comes from the structure of the variables $\xi=t+\theta+h$ and $\eta=t+s+h.$ A similar translation to the new variables is made in cases $(B)-(D).$
\begin{figure}[h]
    \centering
    \includegraphics[width=0.5\linewidth]{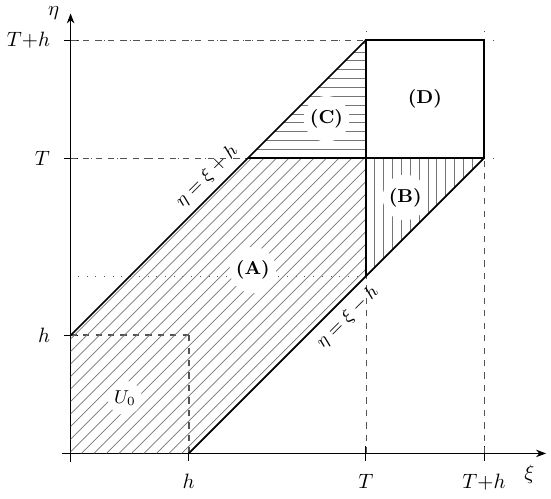}
    \caption{Domain of definition of the matrix $\widetilde{U}$}
    \label{Fig_1}
\end{figure}

Obviously, if $\xi,\eta\in[0,h]$ in $(A)$ then $\widetilde{U}(\xi,\eta) \equiv U_0(\xi,\eta)$ by (\ref{eq_integral_U_0}). 
Hence, the definition given in $(A)$ serves as an extension of the integral equation~(\ref{eq_integral_U_0}). In the sequel, we prove that the function $\widetilde U$ coincides with the delay Lyapunov matrix in the sense of Definition~\ref{def_Lyap_matrix} in all domains $(A)-(D),$ see Remark~\ref{remark_consistence_U_tilde_U} below. Furthermore, we propose a more general way of extending the equation~(\ref{eq_integral_U_0}) to the whole $\R^2$ in the next section. Then, we show that this extension agrees both with the definition of the function $\widetilde{U}$ and with the Definition~\ref{def_Lyap_matrix} of the delay Lyapunov matrix $U.$

\begin{remark}
It follows from the work \cite{Gomez2016} that the delay Lyapunov matrix in the sense of Definition~\ref{def_Lyap_matrix} satisfies both integral properties appeared in $(A)$ and $(B).$ Our theory shows that these properties actually define the delay Lyapunov matrix.
\end{remark}

\subsection{Extension of the integral equation for $U_0$}

Observe that
$[\mathcal L U](\theta,s)$ defines a nontrivial matrix not only for $\theta,s\in[0,T]$ but in principle for all $-\infty<\theta,s\<T$ and is zero if one of the variables is greater than $T.$ This motivates extending the integral equation~(\ref{eq_integral_U_0}) to all values of $\theta,s\<T.$ Specifically, let us consider
\begin{align}
\label{int_eq_extended}
U(\theta,s)  = [\mathcal L U_0 ](\theta,s) + \int_{\theta}^T K^T(\tau,\theta) W(\tau)  K(\tau,s) \dd \tau,\quad \theta,s\<T,
\end{align}
where $U_0$ is a solution to the integral equation~(\ref{eq_integral_U_0}) and $\,\mathcal L\!:\, \mathcal C([0,h]^2,\R^{n\times n})\to \mathcal C((-\infty,T]^2,\R^{n\times n}).$
Obviously, $U(\theta,s)\equiv U_0(\theta,s)$ for $\theta,s\in[0,h].$

Below, we shall prove that the extension~(\ref{int_eq_extended}) of the integral equation~(\ref{eq_integral_U_0}) is meaningful. To this end,
we first study some properties of the matrix function $U$ defined by~(\ref{int_eq_extended}).
\begin{lemma}
    \label{lemma_int_eq_periodicity}
    Equation \textup{(\ref{int_eq_extended})} implies periodicity of the function $U,$ i.e.
    $$
U(\theta-T,s-T)=U(\theta,s)\quad\forall \;\theta,s\< T.
    $$
\end{lemma}

\begin{proof} Let us consider
    \begin{align*}
[\mathcal L U_0] (\theta-T,s-T) &=K^T(T,\theta-T) U_0(0,0) K(T,s-T) \\ &+ K^T(T,\theta-T)\int_{-h}^0 U_0(0,h+\xi)A_1(h+\xi)K(T+\xi,s-T)\dd\xi\notag\\
&+ \int_{-h}^0 K^T(T+\xi,\theta-T)A_1^T(h+\xi)U_0(\xi+h,0)\dd\xi K(T,s-T)\notag\\
&+ \int_{-h}^0\!\int_{-h}^0\! K^T(T+\xi_1,\theta-T)A_1^T(h+\xi_1)U_0(h+\xi_1,h+\xi_2)\\ &\times A_1(h+\xi_2) K(T+\xi_2,s-T) \dd \xi_2\dd \xi_1. 
\end{align*}
To derive a connection between $ [\mathcal L U_0](\theta - T,s-T)$ and $ [\mathcal L U_0](\theta,s),$ we use Lemma~\ref{lemma_aux_fund} written as
\begin{align*}
K(T&+\xi, s-T) = K(T+\xi,0)K(0,s-T)  \\ &+\int_{-h}^0 K(T+\xi,\tau+h)A_1(\tau+h)K(\tau,s-T)\dd\tau\\
&=K(T+\xi,0)K(T,s) + \int_{-h}^0 K(T+\xi,\tau+h)A_1(\tau+h)K(T+\tau,s)\dd\tau,
\end{align*}
where $T+\xi\>0,$ $s-T\<0.$ Let us replace the fundamental matrix in the previous formula using this expression. Then, we expand the brackets and combine the terms of a similar kind. For example, for the terms with $K^T(T,\theta)\ldots K(T,s)$ we get
\begin{align*}
K^T(T,\theta)&\Biggl[K^T(T,0)U_0(0,0)K(T,0) \\ &+ K^T(T,0)\int_{-h}^0 U_0(0,\tau+h)A_1(h+\tau)K(T+\tau,0)\dd\tau \\
&+\int_{-h}^0 K^T(T+\tau,0) A_1^T(h+\tau)U_0(h+\tau,0)\dd\tau K(T,0)\\
&+ \int_{-h}^0\int_{-h}^0 K(T+\tau_1,0)^T A_1^T(h+\tau_1)U_0(h+\tau_1,h+\tau_2)A_1(h+\tau_2)\\ &\times K(T+\tau_2,0)\dd \tau_2\dd \tau_1\Biggr] K(T,s) \\ &=K^T(T,\theta)\Biggl[U_0(0,0) - \int_0^T K^T(\xi,0) W(\xi) K(\xi,0)\dd\xi\Biggr] K(T,s).
\end{align*}
Here, we use the fact that $U_0$ is a solution to the integral equation~(\ref{eq_integral_U_0}).
Combining similarly the remaining terms, we obtain
\begin{align*}
&[\mathcal L U_0] (\theta-T,s-T) = K^T(T,\theta)\Biggl[U_0(0,0) - \int_0^T K^T(\xi,0) W(\xi) K(\xi,0)\dd\xi\Biggr] K(T,s)\\
&+K^T(T,\theta)\int_{-h}^0\Biggl[U_0(0,\tau+h) - \int_0^T\!\!\! K^T(\xi,0) W(\xi) K(\xi,\tau+h)\dd\xi\Biggr]\\ &\times A_1(\tau+h)K(T+\tau,s)\dd\tau
+\int_{-h}^0K^T(T+\tau,\theta)A_1^T(\tau+h)\\ &\times \Biggl[ U_0(\tau+h,0)-\int_0^T\!\!\! K^T(\xi,\tau+h) W(\xi) K(\xi,0)\dd\xi\Biggr]\dd\theta K(T,s)\\
&+\int_{-h}^0 \int_{-h}^0 K^T(T+\tau_1,\theta)A_1^T(\tau_1+h) \Biggl[U_0(\tau_1+h,\tau_2+h)\\ &-\int_0^T K^T(\xi,\tau_1+h) W (\xi)K(\xi,\tau_2+h)\dd\xi\Biggr]A_1(\tau_2+h)K(T+\tau_2,s)\dd\tau_2\dd\tau_1 \\
&=[\mathcal L U_0] (\theta,s) - \widetilde{\Lambda},
\end{align*}
where $\widetilde{\Lambda}$ denotes the rest of the terms. Let us transform $\widetilde{\Lambda}$ using Lemma~\ref{lemma_aux_fund} again:
\begin{align*}
 \widetilde{\Lambda} &= K^T(T,\theta)\int_0^T K^T(\xi,0) W(\xi) \Biggl[ K(\xi,0) K(T,s) \\ &+ \int_{-h}^0 K(\xi,\tau+h)A_1(\tau+h)K(T+\tau,s)\dd\tau\Biggr]\dd\xi \\
 &+\int_{-h}^0 K^T(T+\tau_1,\theta) A_1^T(\tau_1+h) \int_0^T K^T(\xi,\tau_1+h) W(\xi) \\&\times \Biggl[ K(\xi,0) K(T,s)+ \int_{-h}^0 K(\xi,\tau_2+h)A_1(\tau_2+h)K(T+\tau_2,s)\dd\tau_2 \Biggr]\dd\xi\dd\tau_1 \\
 &=K^T(T,\theta)\int_0^T K^T(\xi,0) W(\xi) K(T+\xi,s) \dd\xi\\
 &+\int_{-h}^0 K^T(T+\tau_1,\theta) A_1^T(\tau_1+h) \int_0^T K^T(\xi,\tau_1+h) W(\xi) K(T+\xi,s) \dd\xi\dd\tau_1\\
 &= \int_0^T \Biggl[ K(\xi,0) K(T,\theta) + \int_{-h}^0 K(\xi,\tau_1+h)A_1(\tau_1+h) K(T+\tau_1,\theta) \dd\tau_1 \Biggr]^T\\ &\times W(\xi) K(T+\xi,s) \dd\xi =
 \int_0^T K^T(T+\xi,\theta) W(\xi) K(T+\xi,s) \dd\xi.
\end{align*}
We conclude that
\begin{align*}
U(\theta-T,s-T) &= [\mathcal L U_0] (\theta-T,s-T) + \int_{\theta-T}^{T} K^T(\xi,\theta-T) W(\xi) K(\xi,s-T)\dd \xi \\
&= [\mathcal L U_0] (\theta,s) + \int_{\theta-T}^{0} K^T(T+\xi,\theta) W(\xi) K(T+\xi,s)\dd \xi\\
&= [\mathcal L U_0] (\theta,s) + \int_{\theta}^{T} K^T(\xi,\theta) W(\xi) K(\xi,s)\dd \xi = U(\theta,s).
\end{align*}
Indeed, the function $U$ defined by (\ref{int_eq_extended}) is periodic for any $\theta,s\< T.$
\end{proof}

Next, we study the differential properties of the function $U$ defined by (\ref{int_eq_extended}). To do this, introduce the matrix function
\begin{align*}
G(\theta,s) = K^T(T,\theta)U_0(0,s) + \int_{-h}^0 K^T(T+\tau,\theta)A_1^T(h+\tau)U_0(h+\tau,s)\dd\tau,
\end{align*}
where $\theta\<T,\;s\in[0,h].$
Then, $[\mathcal L U_0] (\theta,s)$ may be rewritten in one of two equivalent ways as
\begin{align}
\label{L_U_s}
[\mathcal L U_0] (\theta,s) &= G(\theta,0)K(T,s) + \int_{-h}^0 G(\theta,\tau+h) A_1(\tau+h) K(T+\tau,s)\dd\tau,
\\
\label{L_U_t}
[\mathcal L U_0] (\theta,s) &= K^T(T,\theta) G^T(s,0) +  \int_{-h}^0 K^T(T+\tau,\theta) A_1^T(\tau+h) G^T(s,h+\tau) \dd\tau
\end{align}
for all $\theta,s\<T.$
Now, the differential properties of the function $U$ defined by (\ref{int_eq_extended}) are summarized in the following lemma.
\begin{lemma}
\label{lemma_derivative_U}
    The matrix function $U(\theta,s)$ defined for $-\infty<\theta,s\<T$ satisfies:

$1)$ for $T-h<\theta\<T,$ $-\infty<s\<T,$
    \begin{align*}
\dfrac{\partial U (\theta,s)}{\partial \theta} &= -A_0^T(\theta) U (\theta,s)- A_1^T(\theta+h) G^T(s,h+\theta-T) - W(\theta) K(\theta,s),
    \end{align*}

$2)$ for $-\infty<\theta\<T-h,$ $-\infty<s\<T,$
\begin{align*}
\dfrac{\partial U (\theta,s)}{\partial \theta} &= -A_0^T(\theta) U (\theta,s)- A_1^T(\theta+h) U(\theta+h,s) - W(\theta) K(\theta,s),
\end{align*}

$3)$ for $T-h<s\<T,$ $-\infty<\theta\<T,$
\begin{align*}
\dfrac{\partial U (\theta,s)}{\partial s} 
&= -U (\theta,s) A_0(s) - G(\theta,s-T+h) A_1(s+h) -K^T(s,\theta)W(s),
\end{align*}

$4)$ for $-\infty<s\<T-h,$ $-\infty<\theta\<T,$
\begin{align*}
\dfrac{\partial U (\theta,s)}{\partial s} 
&= -U (\theta,s) A_0(s) - U(\theta,s+h) A_1(s+h) -K^T(s,\theta)W(s).
\end{align*}
\end{lemma}
\begin{proof} First, let us differentiate the formula~(\ref{L_U_t}) with respect to $\theta$ using expression (\ref{fund_matr_s}) for the  derivative of the fundamental matrix with respect to the second argument.

Case~$1).$ $\theta>T-h.$ In this case, $\theta+h>T\>T+\tau,$ $\tau\in[-h,0],$ therefore the term $K(T+\tau,\theta+h)$ in the derivative of $K$ vanishes, and we are left with
$$
\dfrac{\partial K^T(T+\tau,\theta)}{\partial \theta} = -A_0^T(\theta)K^T(T+\tau,\theta),\quad \tau>\theta-T.
$$
Next, the above formula is applicable for $\tau>\theta-T>-h,$ hence we get
\begin{align}
\dfrac{\partial [\mathcal L U_0] (\theta,s)}{\partial\theta} &= -A_0^T(\theta)K^T(T,\theta) G^T(s,0) -A_0^T(\theta)  \int_{\theta-T}^0 K^T(T+\tau,\theta) A_1^T(\tau+h)\notag\\ &\times G^T(s,h+\tau) \dd\tau -A_1^T(\theta-T+h) G^T(s,h+\theta-T)\notag\\
&=-A_0^T(\theta)[\mathcal L U_0](\theta,s)-A_1^T(\theta+h) G^T(s,h+\theta-T),\label{expr_derivative_L}
\end{align}
where we use the periodicity of $A_1.$

Case~$2).$ $\theta\<T-h.$ In this case, the derivative of the fundamental matrix is equal to
$$
\dfrac{\partial K^T(T+\tau,\theta)}{\partial \theta} = -A_0^T(\theta)K^T(T+\tau,\theta)-A_1^T(\theta+h)K^T(T+\tau,\theta+h),
$$
and this formula is applicable for all $\tau>-h,$ since $-h\>\theta-T.$ In this case we get
\begin{align*}
\dfrac{\partial [\mathcal L U_0] (\theta,s)}{\partial\theta} = &-\Bigl[A_0^T(\theta)K^T(T,\theta)+A_1^T(\theta+h)K^T(T,\theta+h)\Bigr] G^T(s,0) \\
&-  \int_{-h}^0 \Bigl[A_0^T(\theta)K^T(T+\tau,\theta)+A_1^T(\theta+h)K^T(T+\tau,\theta+h)\Bigr] A_1^T(\tau+h)\\ &\times G^T(s,h+\tau) \dd\tau
=-A_0^T(\theta)[\mathcal L U_0](\theta,s) - A_1^T(\theta+h)[\mathcal L U_0](\theta+h,s),
\end{align*}
where $[\mathcal L U_0](\theta+h,s)$ is defined because $\theta+h\<T.$

It remains to differentiate the source term in~(\ref{int_eq_extended}):
\begin{align*}
\dfrac{\partial}{\partial \theta} \int_{\theta}^{T} K^T(\xi,\theta) W(\xi) K(\xi,s)\dd \xi &= -\int_{\theta}^{T} \Bigl[A_0^T(\theta)K^T(\xi,\theta)+ A_1^T(\theta+h)K^T(\xi,\theta+h)\Bigr]\\ &\times W(\xi) K(\xi,s)\dd \xi - W(\theta) K(\theta,s).
\end{align*}
Notice that the term with $K^T(\xi,\theta+h)$ vanishes in the first case when $\theta+h>T\>\xi.$
Gathering all terms, we obtain the desired differential properties of $U$ in cases $1)$ and $2).$ It is essential that equation~(\ref{int_eq_extended}) is assumed true as long as $\theta+h\<T,$ i.e. considering it on $[0,h]^2$ would be not enough to finish the proof. 

Similarly, differentiating the formula~(\ref{L_U_s}) with respect to $s$ and using expression (\ref{fund_matr_t}) for the  derivative of the fundamental matrix with respect to the first argument, one can address the cases $3)$ and $4).$ We only remark that, when differentiating with respect to $s,$ the source term should be represented as
\begin{align*}
\int_\theta^{T} K^T(\xi,\theta) W(\xi) K(\xi,s)\dd \xi = \int_s^{T} K^T(\xi,\theta) W(\xi) K(\xi,s)\dd \xi
\end{align*}
due to the properties of the fundamental matrix.
\end{proof}

The cases $2)$ and $4)$ of Lemma~\ref{lemma_derivative_U} provide us with the ready-made dynamic properties of the matrix $U,$ whereas in cases $1)$ and $3)$ the obtained dynamic properties still depend on the auxiliary function $G.$ It turns out, however, that the $G$ corresponds to the same matrix function $U$ with the first argument shifted. An explicit connection between $U$ and $G$ is established in the next lemma.

\begin{lemma}
    \label{lemma_connection_U_G}
    The function $G$ is connected with $U$ as follows:
$$
G(\theta,s) = U(\theta-T,s)\quad \forall\; \theta\<T,\; s\in [0,h].
$$
\end{lemma}
\begin{proof}
Based on the definition of the function $G,$ we would like to prove that
\begin{align*}
U(\theta-T,s) = K^T(T,\theta)U_0(0,s) + \int_{-h}^0 K^T(T+\tau,\theta)A_1^T(h+\tau)U_0(h+\tau,s)\dd\tau 
\end{align*}
for all $\theta\<T$ and $s\in [0,h].$
To do this, consider the integral equation~(\ref{int_eq_extended}) for $U(\theta-T,s):$
\begin{align}
\label{aux_proof_1}
U(\theta-T,s) &= [\mathcal L U_0](\theta-T,s) + \int_{\theta-T}^{T} K^T(\xi,\theta-T) W(\xi) K(\xi,s)\dd \xi.
\end{align}
Now, we use representation~(\ref{L_U_t}) of $\mathcal L$ to write
\begin{align*}
[\mathcal L U_0] (\theta-T,s) &= K^T(T,\theta-T) G^T(s,0) \\ &+  \int_{-h}^0 K^T(T+\tau,\theta-T) A_1^T(\tau+h) G^T(s,h+\tau) \dd\tau.
\end{align*}
Similarly to the proof of Lemma~\ref{lemma_int_eq_periodicity}, we again use Lemma~\ref{lemma_aux_fund} written as
\begin{align*}
K(T+\tau,\theta-T) &= K(T+\tau,0)K(0,\theta-T) \\ &+ \int_{-h}^0 K(T+\tau,\xi+h)A_1(\xi+h)K(\xi,\theta-T)\dd\xi,
\end{align*}
where $T+\tau\>0$ and $\theta-T\<0.$ Substituting and rearranging the terms, we get the following expression:
\begin{align*}
[\mathcal L U_0] (\theta-T,s) &= K^T(0,\theta-T)\biggl[ K^T(T,0)G^T(s,0) \\+ &  \int_{-h}^0  K^T(T+\tau,0)A_1^T(\tau+h) G^T(s,h+\tau) \dd\tau \biggr]\\
&+\int_{-h}^0 K^T(\xi,\theta-T) A_1^T(\xi+h)  \biggl[ K^T(T,\xi+h) G^T(s,0) \\ &+ \int_{-h}^0  K^T(T+\tau,\xi+h) A_1^T(\tau+h) G^T(s,h+\tau) \dd\tau \biggr]  \dd\xi.
\end{align*}
Using~(\ref{L_U_t}) again, we write
\begin{align*}
[\mathcal L U_0] (\theta-T,s) &= K^T(0,\theta-T) [\mathcal L U_0] (0,s) \\ &+ \int_{-h}^0 K^T(\xi,\theta-T) A_1^T(\xi+h)[\mathcal L U_0] (\xi+h,s)\dd\xi.
\end{align*}
We proceed by using the equation~(\ref{eq_integral_U_0}) since both $\xi+h,s\in[0,h]:$
\begin{align*}
[\mathcal L U_0] (0,s) &= U_0(0,s) - \int_0^{T} K^T(\tau,0) W(\tau) K(\tau,s)\dd \tau,\\
[\mathcal L U_0] (\xi+h,s) &=  U_0(\xi+h,s) -\int_{0}^{T} K^T(\tau,\xi+h) W (\tau)K(\tau,s)\dd \tau.
\end{align*}
Hence,
\begin{align*}
[\mathcal L U_0] (\theta-T,s) &= K^T(0,\theta-T) U_0 (0,s) \\ &+ \int_{-h}^0 K^T(\xi,\theta-T) A_1^T(\xi+h) U_0 (\xi+h,s)\dd\xi-\Psi,\\
\Psi&= \int_{0}^{T} \biggl[ K^T(0,\theta-T) K^T(\tau,0)  \\ &+ \int_{-h}^0 K^T(\xi,\theta-T) A_1^T(\xi+h) K^T(\tau,\xi+h) \dd\xi \biggr]
W(\tau) K(\tau,s)\dd \tau.
\end{align*}
By employing Lemma~\ref{lemma_aux_fund} again and since $\tau\>0,$ we conclude that
\begin{align*}
\Psi&= \int_{0}^{T} K^T(\tau,\theta-T) W(\tau) K(\tau,s)\dd \tau.
\end{align*}
Finally, we substitute the obtained expressions in (\ref{aux_proof_1}):
\begin{align*}
U(\theta-T,s) &= K^T(0,\theta-T) U_0 (0,s) + \int_{-h}^0 K^T(\xi,\theta-T) A_1^T(\xi+h) U_0 (\xi+h,s)\dd\xi \\ &- \int_{0}^{T} K^T(\tau,\theta-T) W(\tau) K(\tau,s)\dd \tau + \int_{\theta-T}^{T} K^T(\tau,\theta-T) W(\tau) K(\tau,s)\dd \tau\\
&=K^T(T,\theta) U_0 (0,s) + \int_{-h}^0 K^T(T+\xi,\theta) A_1^T(\xi+h) U_0 (\xi+h,s)\dd\xi \\
&+\int_{\theta-T}^{0} K^T(\tau,\theta-T) W(\tau) K(\tau,s)\dd \tau,\quad \theta\<T,\quad s\in[0,h].
\end{align*}
It remains to observe that the last integral vanishes for $s\>0.$ The desired expression for $U(\theta-T,s)$ is obtained.
\end{proof}

The results of this section lead us to the following conclusion.

\begin{theorem}
\label{thm_integral_eq}
  Let $U_0\in \mathcal C([0,h]^2,\R^{n\times n})$ be a solution to the integral equation~\textup{(\ref{eq_integral_U_0})} such that $U_0(\theta,s)=U_0^T(s,\theta).$ Then, the matrix function defined as
  \begin{align*}
U(\theta,s)  &= [\mathcal L U_0 ](\theta,s) + \int_{\theta}^T K^T(\tau,\theta) W(\tau)  K(\tau,s) \dd \tau,\quad\text{if}\quad \theta,s\<T,\\
U(\theta,s)  &= U(\theta-T,s-T)\quad \text{otherwise},
\end{align*}
where $\,\mathcal L\!:\, \mathcal C([0,h]^2,\R^{n\times n})\to \mathcal C((-\infty,T]^2,\R^{n\times n}),$
is the delay Lyapunov matrix of system~\textup{(\ref{sys_delay_periodic})} associated with $W(\tau)$ for all $\theta,s\in\R.$
\end{theorem}
\begin{proof}
The symmetry and periodicity properties are obvious from the conditions of the theorem and Lemma~\ref{lemma_int_eq_periodicity}. Next, taking into account Lemma~\ref{lemma_connection_U_G} and the periodicity condition, we note that the dynamic equations in Cases~$1)$ and $2)$ (respectively, $3)$ and $4)$) of Lemma~\ref{lemma_derivative_U} coincide. Hence, the dynamic property of Definition~\ref{def_Lyap_matrix} follows immediately from Cases~$1),$ $2)$ of Lemma~\ref{lemma_derivative_U} considering $s>\theta.$ To verify the algebraic property of Definition~\ref{def_Lyap_matrix}, consider $s\to\theta+0$ and $s\to\theta-0$ in the differential properties of Lemma~\ref{lemma_derivative_U}:
\begin{align*}
\dfrac{\partial U (\theta,s)}{\partial \theta}\Biggl|_{s\to\theta+0} &= -A_0^T(\theta) U (\theta,\theta)- A_1^T(\theta+h) U(\theta+h,\theta),\\
\dfrac{\partial U (\theta,s)}{\partial \theta}\Biggl|_{s\to\theta-0} &= -A_0^T(\theta) U (\theta,\theta)- A_1^T(\theta+h) U(\theta+h,\theta) - W(\theta),\\
\dfrac{\partial U (\theta,s)}{\partial s}\Biggl|_{s\to\theta+0} 
&= -U (\theta,\theta) A_0(\theta) - U(\theta,\theta+h) A_1(\theta+h) -W(\theta),\\
\dfrac{\partial U (\theta,s)}{\partial s} \Biggl|_{s\to\theta-0}
&= -U (\theta,\theta) A_0(\theta) - U(\theta,\theta+h) A_1(\theta+h).
\end{align*}
Writing now 
  $$
\dfrac{\dd U(\theta,\theta)}{\dd\theta} = \lim_{s\to\theta+0}\dfrac{\partial U(\theta,s)}{\partial\theta} + \lim_{s\to\theta+0}\dfrac{\partial U(\theta,s)}{\partial s}= \lim_{s\to\theta-0}\dfrac{\partial U(\theta,s)}{\partial\theta} + \lim_{s\to\theta-0}\dfrac{\partial U(\theta,s)}{\partial s},
  $$
we get the desired ODE property from Definition~\ref{def_Lyap_matrix}.
\end{proof}

On the other hand, it follows from the integral properties presented in \cite{Gomez2016} that the delay Lyapunov matrix in the sense of Definition~\ref{def_Lyap_matrix} satisfies both integral equation~\textup{(\ref{eq_integral_U_0})} and Theorem~\ref{thm_integral_eq}. Hence, Theorem~\ref{thm_integral_eq} may serve as an alternative definition of the delay Lyapunov matrix.

\begin{remark}
\label{remark_consistence_U_tilde_U}
An extension of the integral equation~\textup{(\ref{eq_integral_U_0})} presented in this section is consistent with the definition of the matrix function $\widetilde U$ given in Section~\ref{sec_P_t} when we calculated the value of the operator $\mathcal P (t).$ Indeed, for $\theta,s\in [0,T]$ this is just the same extension, and it remains to justify that $\widetilde U$ defined in the domains $(B)$ and $(C)$ coincides with $U$ defined by Theorem~\ref{thm_integral_eq}. That is, considering $(C)$ we need to show that
\begin{align*}
U(\xi-T,\eta-T)  &= K^T(T,\xi) U_0(0,\eta-T) \\ &+ \int_{-h}^0 K^T(T+\tau,\xi)A_1^T(h+\tau) U_0(h+\tau,\eta-T)\dd\tau,
\end{align*}
where $\xi-T\<0,$ $\eta-T\in[0,h],$ and $U$ is defined by~(\ref{int_eq_extended}). However, the expression on the right-hand side of the above equation is nothing else but $G(\xi,\eta-T).$ Hence, the desired connection has already been proven in Lemma~\ref{lemma_connection_U_G}.
\end{remark}

\begin{remark}
     In view of Theorem~\ref{thm_integral_eq} and Remark~\ref{remark_consistence_U_tilde_U}, we conclude that the Lyapunov functional $\inner{\ph}{\mathcal P(t)\ph}$ given by~(\ref{formula_v_t_R}) which we constructed for $t\in[0,T]$ in Section~\ref{sec_P_t} coincides as expected with the functional (\ref{func_v_0_old_old}) defined through the derivative condition~(\ref{derivative_v_0}) along the solutions of system~(\ref{sys_delay_periodic}).
\end{remark}

\begin{remark}
Originally, the delay Lyapunov matrix was defined in \cite{LetZhab2009, ZhabLet2009IFAC} for exponentially stable systems as
    $$
U(\theta,s)=\int_{\max\{\theta,s\}}^{+\infty} K^T(\xi,\theta) W(\xi) K(\xi,s)\dd\xi.
    $$
The results of this section show that the integral equation~\textup{(\ref{eq_integral_U_0})} together with its extension (\ref{int_eq_extended}) capture a lot of properties of this integral without assuming the exponential stability.
\end{remark}

\subsection{Existence and uniqueness issue for equation~(\ref{eq_integral_U_0})} This section is devoted to studying the existence and uniqueness of solutions to the integral equation~(\ref{eq_integral_U_0}) considered on the Banach space
$$
\mathcal X = \Bigl\{U\in\mathcal C([0,h]^2,\R^{n\times n})\;\Bigl|\; U(\theta,s)=U^T(s,\theta)\quad \forall\,\theta,s\in[0,h]\Bigr\}
$$
supplied with the standard uniform norm.
We start with the following result. 
\begin{lemma}
\label{lemma_L_compact}
    The operator $\mathcal L\!\!: \,\mathcal X\to \mathcal X$ defined by \textup{(\ref{operator_L})} is a compact linear bounded operator on $\mathcal X.$
\end{lemma}

\begin{proof}
    Using the classical Arzelà--Ascoli argument \cite{KolmFomin}, we need to prove that, given a bounded sequence $\{U_k\}_{k\in\mathbb{N}}\subset \mathcal X,$ the corresponding sequence  $\{\mathcal L U_k\}_{k\in\mathbb{N}}$ is uniformly bounded and uniformly equicontinuous. Since in our case the value $[\mathcal L U](\theta,s)$ is differentiable with respect to both variables, see Lemma~\ref{lemma_derivative_U}, it is enough to show that, given a bounded $U\in \mathcal X,$ both $[\mathcal L U](\theta,s)$ itself and its derivatives  with respect to $\theta$ and $s$ are bounded.
    Clearly, there exists a constant $M>0$ such that $\|K(\theta,s)\|\<M$ for $\theta,s\in[0,T].$ Denote $a_0=\max_{t\in\R} \|A_0(t)\|,$ $a_1=\max_{t\in\R} \|A_1(t)\|.$ Then, given a $U\in\mathcal X$ with $\|U(\theta,s)\|\<u$ $(u>0),$ we get $\|G(\theta,s)\|\<M u (1+a_1 h)$ and $\|[\mathcal L U](\theta,s)\|\<M^2 u (1+a_1 h)^2$ for $\theta,s\in[0,h].$ Then, if $\theta>T-h,$ expression (\ref{expr_derivative_L}) allows us to conclude immediately that $\left\|\frac{\partial \mathcal L U}{\partial\theta}\right\|$ is bounded for $\theta,s\in[0,h].$
    If $\theta\in[0,T-h]$ then, although the notation $[\mathcal L U](\theta+h,s)$ is not always well-defined when $\mathcal L$ acts from $\mathcal X$ to $\mathcal X,$ the right-hand side of the formula~(\ref{L_U_t}) is still well-defined, therefore we get 
    \begin{align*}
\left\|\dfrac{\partial [\mathcal L U] (\theta,s)}{\partial\theta}\right\| &\< a_0\|[\mathcal L U](\theta,s)\|+a_1\biggl\|K^T(T,\theta+h) G^T(s,0) \\
&+\int_{-h}^0 K^T(T+\tau,\theta+h) A_1^T(\tau+h) G^T(s,h+\tau) \dd\tau \biggr\|\\
&\<M^2 u (1+a_1 h)^2(a_0+a_1).
\end{align*}
The derivative of $\mathcal L$ with respect to $s$ can be treated similarly.
\end{proof}
The integral equation~(\ref{eq_integral_U_0}) on the Banach space $\mathcal X$ can be written as
\begin{align*}
U_0 &= \mathcal L U_0 + I_W,\quad
I_W(\theta,s) =\int_0^{T} K^T(\xi,\theta) W(\xi) K(\xi,s)\dd \xi.\notag
\end{align*}
An important consequence of Lemma~\ref{lemma_L_compact} is that the Fredholm theory can be applied to this equation \cite{Brezis}.

\begin{corollary} \textup{(\textbf{Fredholm alternative}, \textup{\cite{Brezis}}, Theorem 6.6)} One of the following two alternatives holds.
\label{corollary Fredholm_alternative}

$(i)$ There exists a unique solution $U_0\in\mathcal X$ to equation~\textup{(\ref{eq_integral_U_0})}, equivalently, $1\notin\sigma(\mathcal L).$

$(ii)$ There exists a nontrivial $U_0\in\mathcal X$ such that $U_0 = \mathcal L U_0,$ equivalently, $1\in\sigma(\mathcal L).$ In this case, there is either no or infinitely many solutions to equation~\textup{(\ref{eq_integral_U_0})}.
    \end{corollary}

\section{Existence and uniqueness issue for the delay Lyapunov matrix}
\label{sec_existence_uniqueness}
In this section, we prove the principal result of our work, that is, the fact that the existence and uniqueness of the delay Lyapunov matrix follow naturally from the infinite-dimensional theory developed in Section~\ref{sec_general_Hilbert_setting}, and in particular from  Lemma~\ref{thm_eq_Lyapunov_operator}. An explicit connection between the infinite-dimensional theory and the concept of the delay Lyapunov matrix established in the previous section plays a crucial role for our main results.
Given that the monodromy operator $\mathcal U$ is compact, we start with the following definition which extends naturally a similar concept for time-invariant systems.
\begin{definition}
System~\textup{(\ref{sys_delay_periodic})} is said to satisfy the Lyapunov condition, if there are no $\mu_1,\mu_2\in \sigma(\mathcal U)$ such that $\mu_1\mu_2 = 1.$
\end{definition}
Since the monodromy operator admits the explicit representation~(\ref{operator_monodromy})--(\ref{operator_monodromy_2}), where the matrices $K(t,t_0)$ and $A_1(t)$ are real-valued, we observe that  the Lyapunov condition is equivalent to the condition $1\notin \sigma(\mathcal U)\cdot \overline{\sigma(\mathcal U)}$ appeared in Section~\ref{sec_general_Hilbert_setting}.

We first prove an auxiliary lemma.
\begin{lemma}
\label{lemma_zero_P_zero_U}
Consider the operator $\mathcal P_0: \mathcal H\to\mathcal H,$
$$
\mathcal P_0\ph =
\begin{pmatrix}
    U(0,0)\ph_0 + \int_{-h}^0 U(0,h+\theta)A_1(h+\theta)\Phi(\theta)\dd \theta \\
    A_1^T(h+\cdot)\[U(h+\cdot,0)\ph_0 +  \int_{-h}^0 U(h+\cdot,h+s)A_1(h+s)\Phi(s)\dd s\]
\end{pmatrix},
$$
where $U\in\mathcal C(\R^2,\R^{n\times n})$ is a delay Lyapunov matrix of system~\textup{(\ref{sys_delay_periodic})} associated with $W(\tau)\equiv\mathbb{O}$. Then,
$\mathcal P_0\ph=0$ for all $\ph\in \mathcal H$ implies $U(\theta,s)\equiv 0$ for all $\theta,s\in\mathbb{R}.$
\end{lemma}
\begin{proof}
    The first part of the proof follows the argument given in \cite{LetZhab2009}. Consider a test function with $\ph_0\ne 0,$ and $\Phi(\theta)=0$ for all $\theta\in[-h,0).$
    For this function, $\mathcal P_0\ph = 0$ implies
    \begin{align*}
    U(0,0)\ph_0=0, \quad A_1^T(h+\theta)U(h+\theta,0)\ph_0=0\quad \forall\,\theta\in[-h,0],\quad \forall\,\ph_0\in\C^n.
    \end{align*}
    Since $\ph_0\in\C^n$ is arbitrary, we conclude that
    \begin{align}
    \label{lemma_B_first_concl}
    U(0,0)=\mathbb{O},\quad A_1^T(h+\theta)U(h+\theta,0)=\mathbb{O}\quad\forall\; \theta\in [-h,0].
    \end{align}
    Now, consider a different test function with $\ph_0\ne 0,$ $\Phi(\theta)=\gamma$ for $\theta\in[\tau,\tau+\eps]$ and $\Phi(\theta)=0$ otherwise, where $-h\<\tau<\tau+\eps<0.$ Since we already have (\ref{lemma_B_first_concl}) and using the symmetry property of $U,$ with this test function we immediately get
    \begin{align*}
A_1^T(h+\theta)\int_{\tau}^{\tau+\eps} U(h+\theta,h+s)A_1(h+s)\dd s\gamma = 0\quad \forall\,\gamma\in\C^n,\quad \forall\,\theta\in[-h,0].
    \end{align*}
Since $\gamma\in\C^n$ and $\tau\in[-h,0)$ are arbitrary and by continuity of $U,$ we conclude that
    \begin{align}
    \label{lemma_B_second_concl}
A_1^T(h+\theta)U(h+\theta,h+\tau)A_1(h+\tau)=\mathbb{O}\quad \forall\,\theta,\tau\in[-h,0].
    \end{align}
Next, by condition $U$ is the delay Lyapunov matrix associated with $W(\tau)\equiv\mathbb{O},$ hence equation~(\ref{int_eq_extended}) tells us that
    \begin{align*}
U(\theta,s)  = [\mathcal L U ](\theta,s)\quad\forall\; \theta,s\<T,
\end{align*}
where $\,\mathcal L$ acts from $\mathcal C([0,h]^2,\R^{n\times n})$ to $\mathcal C((-\infty,T]^2,\R^{n\times n}).$
    But, taking into account the structure (\ref{operator_L}) of the operator $\mathcal L$ we observe that both (\ref{lemma_B_first_concl}) and (\ref{lemma_B_second_concl}) immediately imply $[\mathcal L U](\theta,s)\equiv 0,$ and hence $U(\theta,s)\equiv 0$ for all $\theta,s\<T.$ Since $U$ is defined by periodicity for other arguments, we conclude that $U(\theta,s)\equiv 0$ for all $\theta,s\in\mathbb{R}.$
\end{proof}

We are now ready to present the main result of our work.
\begin{theorem}
\label{thm_uniqueness}
Given a continuous $T$-periodic matrix $W(\tau)=W^T(\tau),$ there exists a unique delay Lyapunov matrix of system~\textup{(\ref{sys_delay_periodic})} associated with $W(\tau),$ if and only if the Lyapunov condition holds.
\end{theorem}
The sufficiency of Theorem~\ref{thm_uniqueness} is established in Section~\ref{sec_sufficiency_exist_unique}, while its necessity is proved in
Section~\ref{sec_necessity_exist_unique}. We note that the sufficiency-existence part of Theorem~\ref{thm_uniqueness} has been recently proven in \cite{Egorov2025} by using a completely different approach.

\subsection{Sufficiency part} We first prove the existence and then the uniqueness of the delay Lyapunov matrix.
\label{sec_sufficiency_exist_unique}

    \textit{Existence part.} According to Theorem~\ref{thm_integral_eq}, any solution of the integral equation~(\ref{eq_integral_U_0}) provides us with the delay Lyapunov matrix of system~(\ref{sys_delay_periodic}). Hence, to prove the existence of $U$ we only need to prove the existence of the solution $U_0\in\mathcal X$ to equation~(\ref{eq_integral_U_0}). According to Corollary~\ref{corollary Fredholm_alternative}, it is enough to prove that $1\notin \sigma(\mathcal L).$ 

    Assume, by contradiction, that $1\in\sigma(\mathcal L),$ that is, there exists a nontrivial solution $U_0\in\mathcal X$ to the homogeneous equation
    $$
U_0(\theta,s) = [\mathcal L U_0](\theta,s),\quad \theta,s\in [0,h].
    $$
    Let us construct the operator $\mathcal P(0)$ by formula~(\ref{formula_P_0}) with this choice of $U_0.$ The crucial fact about this $\mathcal P(0)$ is that, according to our argument in Section~\ref{sec_P_0}, it solves by construction the homogeneous operator equation
    $$
\mathcal U^\star \mathcal P(0)\mathcal U - \mathcal P(0) = 0. 
    $$
    However, since the Lyapunov condition holds, Remark~\ref{remark_Lyap_cond} implies that the only solution to the above equation is the trivial one: $\mathcal P(0)\ph = 0$ $\forall\,\ph\in\mathcal H.$ Consider now an extension of the matrix $U_0$ to $\R^2$ given by
\begin{align*}
U(\theta,s)  &= [\mathcal L U_0 ](\theta,s),\quad\text{if}\quad \theta,s\<T,\\
U(\theta,s)  &= U(\theta-T,s-T)\quad \text{otherwise}.
\end{align*}
Theorem~\ref{thm_integral_eq} implies that $U$ is nothing else but the delay Lyapunov matrix associated with $W(\tau)=\mathbb{O}.$ By Lemma~\ref{lemma_zero_P_zero_U} we conclude that $\mathcal P(0)\ph = 0$ $\forall\,\ph\in\mathcal H$ implies $U(\theta,s)\equiv 0$ for all $\theta,s\in\mathbb{R}.$ This contradicts to our assumption that $U_0$ is a nontrivial matrix.
    
\textit{Uniqueness part.} Assume that $U_1(\theta,s)$ and $U_2(\theta,s)$ are two nontrivial delay Lyapunov matrices associated with the same continuous symmetric $T$-periodic matrix $W(\tau).$ Then, they define two quadratic functionals of the form~(\ref{func_v_0_old_old}), $v_1(t,\ph)=\langle\ph, \mathcal P_1(t)\ph \rangle$ and $v_2(t,\ph)=\langle\ph, \mathcal P_2(t)\ph \rangle$
which both solve Problem~1 by construction. However, by
Theorem~\ref{thm_Hilbert} this implies that the corresponding operators
$\mathcal P_1(0)$ and $\mathcal P_2(0)$ both solve the operator equation~(\ref{discr_Lyap}),
\begin{align*}
\mathcal P_i(0) - \mathcal U^\star \mathcal P_i(0) \mathcal U = \int_0^T S^\star(\xi,0)\mathcal W(\xi) S(\xi,0)\dd \xi,\quad i=1,2.
\end{align*}
Since the Lyapunov condition holds, by Remark~\ref{remark_Lyap_cond} we immediately conclude that $\Delta \mathcal P(0) = \mathcal P_1(0)-\mathcal P_2(0)\equiv 0.$ On the other hand, by construction we have
\begin{align*}
\Delta \mathcal P(0)\ph =
\begin{pmatrix}
    \Delta U(0,0)\ph_0 + \int_{-h}^0 \Delta U(0,h+\theta)A_1(h+\theta)\Phi(\theta)\dd \theta \\
    A_1^T(h+\cdot)\left[\Delta U(h+\cdot,0)\ph_0 +  \int_{-h}^0 \Delta U(h+\cdot,h+s)A_1(h+s)\Phi(s)\dd s\right]
\end{pmatrix},
\end{align*}
where $\Delta U(\theta,s) = U_1(\theta,s)-U_2(\theta,s)$ is obviously the delay Lyapunov matrix associated with $W(\tau)\equiv\mathbb{O}.$ Applying now Lemma~\ref{lemma_zero_P_zero_U}, we conclude that $\Delta U(\theta,s)\equiv 0$ for all $\theta,s\in\mathbb{R},$ which proves the uniqueness of the delay Lyapunov matrix.

\subsection{Necessity part}
\label{sec_necessity_exist_unique}
Before proceeding with the proof of the necessity part of Theorem~\ref{thm_uniqueness}, we state several auxiliary results.

\begin{lemma} \textup{(\cite{HaleVerduyn}, Lemma~1.2, p.~237)}
    The value $\mu=e^{\lambda T}$ is a Floquet multiplier of system~\textup{(\ref{sys_delay_periodic})} \textup{(}that is, $\mu\in \sigma(\mathcal U)\setminus\{0\}$\textup{)}, if and only if system~\textup{(\ref{sys_delay_periodic})} admits a solution of the form $e^{\lambda t} p(t),$ where $p(t)$ is a nontrivial $T$-periodic vector function.
\end{lemma}

This lemma is based on the fact that $\psi = \begin{pmatrix}p(0) \\ e^{\lambda\cdot} p(\cdot) \end{pmatrix}$ is nothing else but an eigenfunction of the monodromy operator corresponding to the eigenvalue $\mu=e^{\lambda T},$ that is, $\,\mathcal U \psi = e^{\lambda T}\psi.$ This fact can be directly shown 
using the Cauchy formula and the explicit expression of the monodromy operator through the fundamental matrix.

\begin{corollary}
\label{corollary_Lyap_cond}
System~\textup{(\ref{sys_delay_periodic})} satisfies \textit{the Lyapunov condition}, if and only if  it does not admit two solutions of the form $e^{\lambda t}p(t)$ and $e^{-\lambda t}q(t),$ where $p$ and $q$ are nontrivial $T$-periodic vector functions and $\lambda\in\mathbb{C},$ at the same time.
\end{corollary}
Indeed, for any $\mu_1,\mu_2\in \sigma(\mathcal U)$ such that $\mu_1\mu_2 = 1$ there are two values $\lambda_1,\lambda_2\in\mathbb{C}$ such that $\mu_1=e^{\lambda_1 T},$ $\mu_2=e^{\lambda_2 T},$ and $\lambda_1+\lambda_2=0.$

Next, in order to accomplish the necessity proof, let us consider
a dual system to system~(\ref{sys_delay_periodic}) introduced in \cite{MichGom2020}:
\begin{equation}
\label{sys_delay_periodic_dual}
\dot z(t) =  A_0^T(-t)z(t) + A_1^T(-t+h)z(t-h).
\end{equation}
It is convenient to reverse the time as $\tau=-t,$ $\zeta(\tau)=z(t).$ By doing this, we rewrite the dual system as
\begin{equation}
\label{sys_dual_changed}
\dot \zeta(\tau) =  - A_0^T(\tau)\zeta(\tau) - A_1^T(\tau+h)\zeta(\tau+h),\quad \tau\<0.
\end{equation}
The following lemma is crucial:
\begin{lemma} \textup{\cite{MichGom2020}}
\label{lemma_spectra_dual}
    Let $\mathcal U$ and $\mathcal U_D$ be the monodromy operators corresponding to systems~\textup{(\ref{sys_delay_periodic})} and \textup{(\ref{sys_delay_periodic_dual})}, respectively. Then, their spectra satisfy
    $$\sigma(\mathcal U)\setminus \{0\} = \sigma(\mathcal U_D)\setminus \{0\}.$$
\end{lemma}
Lemma~\ref{lemma_spectra_dual} implies that system~(\ref{sys_delay_periodic}) satisfies the Lyapunov condition, if and only if the dual system~(\ref{sys_delay_periodic_dual}) does. 
We are now ready to proceed with the proof of Theorem~\ref{thm_uniqueness}.

\textit{Proof of the necessity part of Theorem~\ref{thm_uniqueness}}.
Assume that for any continuous $T$-periodic $W(\tau)=W^T(\tau)$ there exists a unique delay Lyapunov matrix of system~\textup{(\ref{sys_delay_periodic})} associated with $W(\tau)$ but the Lyapunov condition is violated. By Lemma~\ref{lemma_spectra_dual}, the Lyapunov condition is violated for the dual system~(\ref{sys_delay_periodic_dual}), too. We use an equivalent definition of the Lyapunov condition provided by Corollary~\ref{corollary_Lyap_cond}. Moreover, it is easy to see that this equivalent formulation also applies to the system~(\ref{sys_dual_changed}). That is, violation of the Lyapunov condition implies that system~(\ref{sys_dual_changed}) has two solutions of the form $\zeta_1(\tau) = e^{\lambda \tau}p(\tau)$ and $\zeta_2(\tau) = e^{-\lambda \tau}q(\tau),$ where $p(\tau+T)\equiv p(\tau),$ $q(\tau+T)\equiv q(\tau).$

We substitute these functions in the system equation~(\ref{sys_dual_changed}) and get the following equations that by construction are satisfied by the periodic functions $p$ and $q:$
\begin{align}
\label{eq_p_q}
 p'(\tau) &= -(\lambda I + A_0^T(\tau))p(\tau) - A_1^T(\tau+h)e^{\lambda h} p(\tau+h),\\
 q'(\tau) &= (\lambda I - A_0^T(\tau)) q(\tau) - A_1^T(\tau+h)e^{-\lambda h} q(\tau+h).\notag
\end{align}
Note that, although in principle we defined the dual system~(\ref{sys_dual_changed}) for $\tau\<0,$ the particular solutions $\zeta_1(\tau)$ and $\zeta_2(\tau)$ are defined for all $\tau.$ Moreover, equations~(\ref{eq_p_q}) imply that they also  satisfy system~(\ref{sys_dual_changed}) for all $\tau.$

Now, let us construct the matrix
\begin{align}
\label{U_W_0_violation}
\widetilde{U}(\theta,s) = \dfrac{1}{2}\Bigl(e^{\lambda(\theta-s)}p(\theta)q^T(s)+ e^{\lambda(s-\theta)}q(\theta) p^T(s)\Bigr),\quad \theta,s\in\R.
\end{align}

In the remaining part of the proof we show that the matrix (\ref{U_W_0_violation}) is the delay Lyapunov matrix of system~(\ref{sys_delay_periodic}) associated with $W(\tau)\equiv\mathbb{O}.$
Obviously, this matrix satisfies the symmetry and the periodicity conditions by construction:
$$
\widetilde{U}(\theta,s)=\widetilde{U}^T(s,\theta),\quad \widetilde{U}(\theta+T,s+T)=\widetilde{U}(\theta,s).
$$
Let us now compute its partial derivative with respect to $\theta:$
\begin{align}
\notag\dfrac{\partial \widetilde{U}(\theta,s)}{\partial \theta} &= \dfrac{\lambda}{2} e^{\lambda(\theta-s)}p(\theta)q^T(s) + \dfrac{1}{2} e^{\lambda(\theta-s)}\Bigl(-(\lambda I + A_0^T(\theta))p(\theta) \\ \notag &- A_1^T(\theta+h)e^{\lambda h} p(\theta+h)\Bigr)q^T(s) 
 -\dfrac{\lambda}{2} e^{\lambda(s-\theta)}q(\theta) p^T(s) \\ \notag &+ \dfrac{1}{2}e^{\lambda(s-\theta)} \Bigl((\lambda I - A_0^T(\theta)) q(\theta) - A_1^T(\theta+h)e^{-\lambda h} q(\theta+h)\Bigr) p^T(s)\\
 &= - A_0^T(\theta) \widetilde{U}(\theta,s) - A_1^T(h+\theta)\widetilde{U}(h+\theta,s). \label{der_U_tilde_theta}
\end{align}
Similarly, we compute the derivative of the matrix (\ref{U_W_0_violation}) with respect to $s:$
\begin{align}
\label{der_U_tilde_s}
\dfrac{\partial \widetilde{U}(\theta,s)}{\partial s}
= -\widetilde{U}(\theta,s)A_0(s) - \widetilde{U}(\theta,h+s) A_1(h+s).
\end{align}
Equation~(\ref{der_U_tilde_theta}) coincides with the PDE property of the delay Lyapunov matrix. Similarly to the proof of Theorem~\ref{thm_integral_eq}, equations (\ref{der_U_tilde_theta}) and (\ref{der_U_tilde_s}) together imply the ODE property of the delay Lyapunov matrix with $W(\tau)\equiv\mathbb{O}.$
We conclude that the matrix function $\widetilde{U}(\theta,s)$ defined by~(\ref{U_W_0_violation}) is indeed a nontrivial delay Lyapunov matrix associated with $W(\tau)\equiv\mathbb{O}.$ Clearly, this implies that, given a delay Lyapunov matrix $U(\theta,s)$ associated with $W(\tau),$ the matrix $U(\theta,s)+\widetilde{U}(\theta,s)$ is a different delay Lyapunov matrix associated with $W(\tau),$ which contradicts to the uniqueness of the delay Lyapunov matrix. $\hfill\Box$

\section{Conclusion}
In this work, an explicit connection between the periodic Lyapunov equation on a Hilbert space and the delay Lyapunov matrix framework for linear periodic time-delay systems is established. This connection allows us proving that the delay Lyapunov matrix of system~(\ref{sys_delay_periodic}) is unique for any $W,$ if and only if its monodromy operator does not have reciprocal eigenvalues, which is an analogue of the Lyapunov condition in a time-invariant case. As a by-product we show that the periodic Lyapunov equation on a Hilbert space allows constructing the functionals with prescribed derivatives for linear periodic delay systems without a preliminary exponential stability assumption. 
We note that the construction of delay Lyapunov matrices for periodic systems still remains a non-trivial problem, especially in the case when $T\ne h.$ Some recent developments on the construction issue for the case $T=h$ can be found in \cite{AlVel_Lyap_matr_scalar,AlVel_Lyap_matr_vector}. 

\section*{Acknowledgments}
The \LaTeX{} code for Figure~\ref{Fig_1} was generated with the assistance of Gemini and Cursor Composer~2.

\bibliographystyle{siamplain}

\end{document}

%% file: ex_shared.tex
\usepackage{lipsum}
\usepackage{amsfonts}
\usepackage{graphicx}
\usepackage{epstopdf}
\usepackage{algorithmic}
\ifpdf
  \DeclareGraphicsExtensions{.eps,.pdf,.png,.jpg}
\else
  \DeclareGraphicsExtensions{.eps}
\fi

\newsiamremark{remark}{Remark}
\newsiamremark{hypothesis}{Hypothesis}
\crefname{hypothesis}{Hypothesis}{Hypotheses}
\newsiamthm{claim}{Claim}
\newsiamremark{fact}{Fact}
\crefname{fact}{Fact}{Facts}

\headers{Delay Periodic Lyapunov Equation}{I.V. Aleksandrova, and J.J.L. Vel\'azquez}

\title{Delay Periodic Lyapunov Equation\thanks{This study was not presented at any conference.\funding{This work was supported by the German Science Foundation (DFG, Deutsche Forschungsgemeinschaft), project 552846308, and by the Germany’s Excellence Strategy-EXC2047/2-390685813 funded by DFG.}}}

\author{Irina V. Aleksandrova\thanks{Institute for Applied Mathematics, University of Bonn, Endenicher Allee 60, D-53115, Bonn, Germany 
  (\email{aleksandrova@iam.uni-bonn.de}).}
\and Juan J.L. Vel\'azquez\thanks{Institute for Applied Mathematics, University of Bonn, Endenicher Allee 60, D-53115, Bonn, Germany
  (\email{velazquez@iam.uni-bonn.de}).}
  }

\usepackage{amsopn}
